\documentclass[12pt,reqno]{amsart}
\usepackage{amssymb,amscd,verbatim,latexsym}
\usepackage[all]{xy}
\SelectTips{eu}{11}
\usepackage{mathrsfs}

\swapnumbers
\newtheorem{thm}[subsection]{Theorem}
\newtheorem{propose}[subsection]{Proposition}
\newtheorem{lemma}[subsection]{Lemma}

\newtheorem{cor}[subsection]{Corollary}
\theoremstyle{definition}
\newtheorem{defn}[subsection]{Definition}
\newtheorem{remark}[subsection]{Remark}

\numberwithin{equation}{section}

\renewcommand{\d}{\mbox{\LARGE $\cdot $}}

\newcommand{\Gal}{{\rm Gal}\,}       
\newcommand{\Lie}{{\rm Lie}\,}       
\newcommand{\Spec}{\operatorname{Spec}} 
\newcommand{\Hom}{\operatorname{Hom}}      
\newcommand{\RHom}{\operatorname{RHom}}    
\newcommand{\DM}{{\bf DM}}          
\newcommand{\M}{{\sf EHM}}   
\newcommand{\Rmod}{\text{$R$-{\sf mod}}}
\newcommand{\Kmod}{\text{$K$-{\sf mod}}}
\newcommand{\Zmod}{\text{$\Z$-{\sf mod}}}
\newcommand{\Qmod}{\text{$\Q$-{\sf mod}}}

\newcommand{\End}{\operatorname{End}}      

\newcommand{\Aff}{\mathbb{A}}   

\newcommand{\C}{\mathbb{C}}     
\newcommand{\Q}{\mathbb{Q}}     
\newcommand{\Z}{\mathbb{Z}}     
\newcommand{\N}{\mathbb{N}}
\newcommand{\G}{\mathbb{G}}     

\renewcommand{\ker}{\operatorname{Ker}}  
\newcommand{\coker}{\operatorname{Coker}} 
\newcommand{\Pic}{{\rm Pic}}     
\newcommand{\Alb}{{\rm Alb}}     
\newcommand{\LAlb}{{\rm LAlb}}     
\newcommand{\LA}[1]{\mbox{${\rm L}_{#1}{\rm Alb}$}}

\newcommand{\Tot}{{\rm Tot}}     
\newcommand{\NS}  {{\rm NS}}      

\newcommand{\longby}[1]{\stackrel{#1}{\longrightarrow}}

\renewcommand{\tilde}{\widetilde}
\newcommand{\df}{\mbox{\,${:=}$}\,}
\newcommand{\ie}{{\it i.e.\/},\ }
\newcommand{\cf}{{\it cf.\/}\ }
\newcommand{\eg}{{\it e.g.\/},\ }
\newcommand{\et} {\textrm{\'et}}

\newcommand{\an}{\textrm{an}}
\newcommand{\eff}{{\rm eff}}
\newcommand{\gm}{{\rm gm}}

\renewcommand{\bar}{\overline}
\newcommand{\into}{\hookrightarrow}

\newcommand{\colim}[1]{{\mathop{\rm
Colim}_{\buildrel\over{#1}}}\;}
\renewcommand{\lim}[1]{{\mathop{\rm
Lim}_{\buildrel\over{#1}}}\;}

\newcommand{\boxtensor}{\def\boxtimesten{\Box\kern-7.59pt\raise1.2pt
\hbox{$\times$} }}                                  

\newcounter{elno}                   

\newcommand{\cC}{\mathcal{C}}

\newcommand{\cE}{\mathcal{E}}
\newcommand{\cF}{\mathcal{F}}
\newcommand{\cG}{\mathcal{G}}

\newcommand{\cM}{\mathcal{M}}
\newcommand{\cN}{\mathcal{N}}

\newcommand{\cR}{\mathcal{R}}

\begin{document}

\title{Nori $1$-motives}
\author{Joseph Ayoub}
\address{\newline
Institut f\"ur Mathematik\\ Universit\"at Z\"urich\\Winterthurerstr.~ 190\\ CH-8057 Z\"urich\\ 
Switzerland 
\newline 
\& CNRS\\ LAGA Universit\'e Paris 13\\ 99 avenue J.B.~Cl\'ement \\ 93430 Villetaneuse\\ France}
\email{joseph.ayoub@math.uzh.ch}
\thanks{The first author was supported in part by the Swiss National Science Foundation (NSF), Grant
No.~2000201-124737/1. The second author is grateful to IH\'ES for hospitality and excellent working conditions.}
\author{Luca Barbieri-Viale}
\address{\newline
Dipartimento di Matematica ``F. Enriques", Universit{\`a} degli Studi di Milano\\ Via C. Saldini 50\\ I-20133 Milano\\ Italy}
\email{luca.barbieri-viale@unimi.it}
\date{May, 2012} 

\begin{abstract} \begin{sloppypar} 
Let $\M$ be Nori's category of effective homological 
mixed motives. In this paper, we consider the 
thick abelian subcategory 
$\M_1\subset \M$ generated by the $i$-th relative homology of pairs of varieties 
for 
$i\in \{0,1\}$. 
We show that $\M_1$ is naturally equivalent to 
the abelian category ${}^t\mathcal{M}_1$ of $1$-motives with torsion;
this is our main theorem. 
Along the way, we obtain several interesting results. Firstly,
we realize
${}^t\mathcal{M}_1$ as the universal abelian category obtained, 
using Nori's formalism, from 
the Betti representation of an explicit diagram 
of curves.  
Secondly, we obtain a conceptual proof of a theorem of Vologodsky on realizations of $1$-motives.
Thirdly, we verify a 
conjecture of Deligne on extensions of $1$-motives in the category of mixed realizations for those extensions that are effective in 
Nori's sense. 
\end{sloppypar}
\end{abstract}
\maketitle
\tableofcontents
\newpage

\section{Introduction}

Nori's construction of the category of effective homological 
mixed motives $\M$ is drafted in \cite{NN}, \cite{HS},
\cite{LV} and generalized in \cite{AR} (see also \cite{BN} for the abstract part of the construction). For details on $\M$ we here mainly refer to \cite{LV} (which is published and consistent with Nori's notation in \cite{NN}). 
The category $\M$ depends on a base field $k$
together with a complex embedding $\sigma:k\hookrightarrow \C$. 
Given a $k$-scheme $X$ and 
a closed subset $Y\subset X$, we have motives 
$\widetilde{H}_i(X,Y;\Z) \in \M$ for all $i\in \Z$.
Their Betti realisations are the usual relative homology groups
$H_i(X(\C),Y(\C);\Z)$.

For $n\in \N$, we denote 
$\M_n\subset \M$ the thick abelian subcategory 
generated by 
$\tilde{H}_i(X,Y;\Z)$
for all $i\leq n$ and all pairs $(X,Y)$ consisting of a 
$k$-scheme $X$ and a closed subset $Y\subset X$. 
Objects of
$\M_n$ will be called Nori $n$-motives.

We can well describe $\M_n$ for $n \leq 1$.
Quite straightforwardly, $\M_0$ is  
equivalent to the category of $0$-motives (see Proposition \ref{N0}).
The case $n=1$ is much more difficult, but nevertheless we are able to 
show the expected property: 
$\M_1$ is equivalent to the abelian category of $1$-motives with torsion; this is our main result 
(see Theorem \ref{Del=Nori}). 
Its proof relies on two preliminary results 
which are of independent interest. 
The first result (see Theorem \ref{thm:main-thm}) claims that the category of $1$-motives with 
torsion is the universal category, in the sense of Nori, 
obtained from the 
Betti representation of an 
explicit diagram of curves.
The second result (see Theorem
\ref{thm:eff-Del-conj}) is a particular case of 
a conjecture of Deligne \cite[2.4]{DE}. Roughly speaking, we prove that a mixed realization which is an extension of $1$-motives is itself a $1$-motive provided that it is \emph{effectively} coming from 
geometry, \ie is the mixed realization of a 
Nori \emph{effective} motive. 
(The original conjecture of Deligne would predict that the same
holds provided that the extension is the realization of a non necessarily effective Nori motive.)

Also, in the course of proving our main theorem, we obtain a conceptual proof of a 
result of Vologodsky comparing two 
Hodge realizations on the category of $1$-motives: 
the classical 
one constructed by Deligne, and the composition of 
Huber's Hodge realization with
the embedding of $1$-motives 
into Voevodsky's triangulated motives (see
Remark \ref{rem:new-volog}).
Finally, with rational coefficients, we are also able to 
construct a left adjoint to the inclusion 
$\M_1\hookrightarrow \M$ (see 
Theorem \ref{left-adj-EHM}).

Anyways, note that a comparison between the existing abelian category of $2$-motives proposed  in \cite{AY} and $\M_2$ is already far beyond the actual techniques.

\subsection*{\it Notation and Conventions} We let $k$ be our base field and tacitly fix an embedding $\sigma: k\into \C$. We denote 
$Sch_k$ the category of $k$-schemes. 
(By $k$-scheme we always mean a finite type, separated and reduced 
$k$-scheme.)
Given a $k$-scheme $X$, we denote $X^{\an}$ the associated analytic space given by $X(\C)$.  We fix a Noetherian commutative ring $R$ and 
denote $\Rmod$ the category of finitely generated $R$-modules.
(In practice, $R$ will be $\Z$ or $\Q$.)
If $E$ is an $R$-algebra which is not necessary commutative, 
we also denote $E\text{-{\sf mod}}$ the category of left $E$-modules 
which are finitely generated over $R$. (This will not be confusing: 
we only use this when $E$ is 
finitely generated as an $R$-module, in which case 
a left $E$-module is finitely generated if and only if it is so 
as an $R$-module.) 
We say that a 
functor $\cE\to \Rmod$ from an $R$-linear abelian category 
$\cE$ is \emph{forgetful}  
if it is $R$-linear, faithful and exact.
(For instance, the obvious functor 
$E\text{-{\sf mod}} \to \Rmod$ is forgetful.)
A \emph{thick abelian} subcategory is a full subcategory, containing 
the zero object, and
stable under sub-quotients and extensions. 
A \emph{thick triangulated} subcategory 
is a full subcategory, containing the zero object, and 
stable under cones, desuspensions and direct summands.

\section{Nori's universal category}

We recall Nori's construction of a universal abelian category from 
a representation of a diagram. In some situation, 
we give a 
characterization of this universal category.

\subsection{Generalities}
Recall the following construction due to Nori (see \cite{BN} and 
\cite[\S 5.3.3]{LV} for details). 
Given a representation $T: D \to \Rmod$ of a diagram $D$, 
there are an $R$-linear abelian category $\cC(T)$, 
a forgetful functor 
$F_T: \cC (T) \to \Rmod$ and a representation
$\tilde T :D \to  \cC (T)$ such that $F_T\circ\tilde T = T$.
Moreover, the triple $(\tilde{T},\cC(T),F_T)$ is initial 
(up to isomorphisms of functors) among factorizations of the representation $T$ as a representation to 
an $R$-linear abelian category followed by 
a forgetful functor. For the precise statement, we refer the 
reader to \cite[Theorem 41]{LV}.

\subsection{\/} When $D$ is finite, one takes
$\cC(T)=\End(T)\text{-{\sf mod}}$ where $\End(T)$ is the 
$R$-algebra of endomorphisms of $T$. More precisely,
an element of $\End(T)$ is a family 
$$(a_p)_{p\in Ob(D)}\in 
\prod_{p\in Ob(D)}\End(T(p))$$
such that for every arrow $a:p\to q$, one has 
$T(a)\circ a_p=a_q\circ T(a)$. For $p\in Ob(D)$, the algebra 
$\End(T)$ acts on the left on $T(p)$. The resulting left
$\End(T)$-module will be denoted by 
$\widetilde{T}(p)$. This gives the representation 
$\widetilde{T}:D\to \End(T)\text{-}\mathsf{mod}$.
We clearly have $T=F_T\circ \widetilde{T}$ where $F_T$ is the 
obvious forgetful functor. 

\subsection{\/} When the diagram is no longer assumed to be finite, we set
\begin{equation}
\label{C(T)=colim}
\cC (T)\df 2-\colim{E \subset D} \End (T|_{F})\text{-{\sf mod}}
\end{equation}
taking the colimit over all finite sub-diagrams $E$. 
The universality of 
$\mathcal{C}(T)$ is established in \cite{BN}. From this description, 
it follows that if $D$ is a filtered 
union of a family of (non necessarily finite) sub-diagrams $(D_{\alpha})_{\alpha\in I}$, there is an equivalence of 
categories
\begin{equation}
\label{colim-sim-C(T)}
2-\colim{\alpha\in I} \cC(T|_{D_{\alpha}})\simeq 
\mathcal{C}(T).
\end{equation}

\subsection{\/} Assume that $R$ is an integral domain 
with field of fractions $K$. 
Denote $T_K:D\to \Kmod$ the representation 
defined by $T_K(-)=T(-)\otimes_R K$
Then, there is a canonical 
equivalence of categories
\begin{equation}\label{chg-R-from-Z-to-Q}
\mathcal{C}(T)\otimes_R K\simeq 
\mathcal{C}(T_K).
\end{equation}
(For an $R$-linear category $\mathcal{A}$, 
the category $\mathcal{B}=\mathcal{A}\otimes_R K$ is given by 
$Ob(\mathcal{B})=Ob(\mathcal{A})$ and 
$\hom_{\mathcal{B}}(-,-)=\hom_{\mathcal{A}}(-,-)\otimes_RK$.)
To check this, it is enough, by 
\eqref{C(T)=colim}, to consider the case where 
$D$ is finite. As $K$ is a flat $R$-algebra, we have 
${\rm End}(T_K)={\rm End}(T)\otimes_R K$.
This easily implies the equivalence 
\eqref{chg-R-from-Z-to-Q}.

\subsection{\/}
\label{subsec-artinian}
Assume that the ring $R$ is artinian (\eg a field).
Then the pro-system $\{\End(T|_E)\}_{E\subset D}$, where $E$ 
runs over the ordered set of finite sub-diagrams of $D$, satisfies 
the Mittag-Leffler condition. Thus, it is tempting 
to consider
\begin{equation}
\label{end-T}
{\rm End}(T)\df\lim{E\subset D} {\rm End}(T|_E)
\end{equation}
endowed with the inverse limit topology. 
Then $0\in {\rm End}(T)$ has a fundamental system of neighborhoods 
consisting of open and closed two-sided ideals $I$ such that 
${\rm End}(T)/I$ 
is a finite length $R$-module.
We denote 
${\rm End}(T)\text{-}\mathsf{mod}$
the category of continuous ${\rm End}(T)$-modules 
which are of finite type over $R$ and discrete, \ie 
annihilated by an open and closed  
$2$-sided ideal of $\End(T)$. 
It follows immediately that there are equivalences of categories
(\cf \cite[\S 1.2.1]{NN})
\begin{equation}
\label{C(T)-if-R-artin}
\begin{array}{rcl}\mathcal{C}(T) & \simeq 
& 2-\colim{I} {\rm End}(T)/I\text{-}\mathsf{mod}\\
&  \simeq & {\rm End}(T)\text{-}\mathsf{mod}.
\end{array}
\end{equation}
(The colimit above is over the open and closed two sided 
ideals in ${\rm End}(T)$.)

\subsection{Criterion for an equivalence}
\label{subsec-criterion}
We keep the notation as in the previous paragraph. Assume that we 
are given a representation $S:D\to \cE$ into an $R$-linear 
abelian category
$\cE$ and a forgetful functor $G:\cE\to \Rmod$ such that 
$G\circ S=T$. By universality (\ie
\cite[Theorem 41]{LV})
there is an exact faithful $R$-linear functor 
\begin{equation}
\label{eqn:criter-equiv}
U:\cC(T)\to \cE
\end{equation}
such that $F_T=G\circ U$ and $S=U\circ \widetilde{T}$.

\begin{propose}
\label{prop:criterion-equi}
Assume the following conditions:
\begin{enumerate}

\item[(a)] Given $p,\,p'\in Ob(D)$, there exist 
$p\sqcup p'\in Ob(D)$, and arrows 
$i: p\to p\sqcup p'$ and 
$i': p'\to p\sqcup p'$ such that 
$$T(i)+T(i'): T(p)\oplus T(p')\to T(p\sqcup p')$$
is an isomorphism.

\item[(b)] Every object of $\cE$ is a quotient of an object of the form
$S(p)$ with $p\in Ob(D)$.

\item[(c)] For every map $S(p)\to A$ in $\cE$ there exists a 
finite sub-diagram $E\subset D$ containing $p$ such that 
$$\ker\{T(p)=G\circ S(p)\to G(A)\}$$ 
is a sub-$\End(T|_E)$-module of $T(p)$.

\end{enumerate}
Then $U$ is an equivalence of categories.
\end{propose}

We first note the following lemma.

\begin{lemma}
\label{lem:S-p-p-prime}
Assume that \ref{prop:criterion-equi}(a) is satisfied and 
that \ref{prop:criterion-equi}(c) holds when 
$A=S(q)$ with $q\in D$.
For every map $f:S(p)\to S(p')$ in 
$\cE$, there exists a 
finite sudiagram $E\subset D$ containing $p$ and $p'$ such that 
$G(f):G\circ S(p)\to G\circ S(p')$
is a morphism of 
$\End(T|_E)$-modules.
\end{lemma}

\begin{proof}
Consider the morphism
$$g=f-{\rm id}_{S(p')}:S(p)\oplus S(p')\to S(p').$$
It suffices to show that 
$\ker(G(g))\subset T(p)\oplus T(p')$ 
is a sub-$\End(T|_E)$-module for 
some finite sub-diagram $E\subset D$ containing 
$p$ and $p'$.

With the notation as in 
\ref{prop:criterion-equi}(a),
the morphism 
$$S(i)+S(i'):S(p)\oplus S(p')\to S(p\sqcup p')$$
is invertible as 
$G$ is faithful and exact.
Moreover, the isomorphism
$$G\circ S(i)+G\circ S(i'):
G\circ S(p)\oplus G\circ S(p')\simeq G\circ S (p\sqcup p')$$
is a morphism of $\End(T|_E)$-modules for any finite diagram $E$ 
containing the objects $p$, $p'$ and $p\sqcup p'$, and the arrows $i$ and $i'$. (Indeed, this 
coincides with the isomorphism
$T(i)+T(i'):T(p)\oplus T(p')\simeq T(p\sqcup p')$.)
Hence, it suffices to show that 
the kernel of the map 
$G\circ S(p\sqcup p')\to G\circ S(p')$
is a sub-$\End(T|_E)$-module for some finite sub-diagram $E$.
This is granted by \ref{prop:criterion-equi}(c)
for $A=S(p')$.
\end{proof}

\subsection{Proof of Proposition \ref{prop:criterion-equi}}
We first construct a functor 
$V:\cE\to \mathcal{C}(T)$. 
For $A\in \cE$, 
we choose 
an epimorphism 
$\alpha:S(p)\twoheadrightarrow A$ in $\cE$
with $p\in Ob(D)$. (For this, we
use \ref{prop:criterion-equi}(b).)
Let $E\subset D$ be a finite sub-diagram such that 
the kernel of $G(\alpha)$ is a sub-${\rm End}(T|_E)$-module,
\ie $E$ is as in
\ref{prop:criterion-equi}(c).
As $T(p)=G\circ S(p)\to G(A)$
is surjective, there is a unique structure of 
$\End(T|_E)$-module on $G(A)$ such that 
$G(\alpha)$ is $\End(T|_E)$-linear.
This $\End(T|_E)$-module defines an objet of 
$\mathcal{C}(T)$
which we denote $V(A,\alpha)$.
By construction, $F_T(V(A,\alpha))=G(A)$.

Next consider a commutative square in $\cE$:
$$\xymatrix{S(p)\ar@{->>}[r]^-{\alpha} \ar[d]^f & A \ar[d]^-e\\
S(p') \ar@{->>}[r]^-{\alpha'} & A'.\!}$$
Using Lemma \ref{lem:S-p-p-prime}, the map 
$G\circ S(p)\to G\circ S(p')$ is 
a morphism of $\End(T|_E)$-modules for 
some finite diagram $E\subset T$.
Enlarging $E$ so that the kernels of $G(\alpha)$ and 
$G(\alpha')$ are sub-${\rm End}(T|_E)$-modules, we get that 
$G(e):G(A) \to G(A')$ is $\End(T|_E)$-linear. 
This gives a morphism
$$V(e,f,\alpha,\alpha'):V(A,\alpha)\to V(A',\alpha')$$ 
in $\cC(T)$. By construction, 
$F_T(V(e,f,\alpha,\alpha'))=G(e)$.

We are now ready to define the functor 
$V$. First, we note that 
$V(A,\alpha)$ is independent of 
$\alpha$.
More precisely, given another epimorphism 
$\alpha':S(p')\twoheadrightarrow A$, there is a 
unique isomorphism $v_{\alpha,\alpha'}:V(A,\alpha)\simeq V(A,\alpha')$ 
such that $F_T(v_{\alpha,\alpha'})={\rm id}_{G(A)}$.
It is given by the composition of
$$V(A,\alpha)\overset{\sim}{\rightarrow}
V(A,\alpha+\alpha') 
\overset{\sim}{\leftarrow}
V(A,\alpha')$$
where the maps are  
$V({\rm id}_A,S(i),\alpha, \alpha+\alpha')$
and 
$V({\rm id}_A,S(i'),\alpha', \alpha+\alpha')$
with $i:p\to p\sqcup p'$ and 
$i':p'\to p\sqcup p'$ as in 
\ref{prop:criterion-equi}(a).

Similarly, given a morphism $e:A\to A'$, 
there is a unique
morphism $V(e):V(A,\alpha)\to V(A,\alpha')$
such that $F_T(V(e))=G(e)$. It is given by the 
composition of 
$$V(A,\alpha)\to V(A',\alpha+\alpha')
\overset{\sim}{\leftarrow}
V(A',\alpha')$$
where the maps are 
$V(e,S(i),\alpha, \alpha+\alpha')$
and $V({\rm id}_A,S(i'),\alpha', \alpha+\alpha')$.
Hence, choosing for every 
$A$ an epimorphism $\alpha_A$ yields a functor 
$V:\cE\to  \cC(T)$ such that 
$F_T\circ V=G$.

Now, as $G$ and $F_T$ are forgetful functors, 
it is immediate that 
$V$ is an $R$-linear faithful and exact functor.
Moreover, choosing the epimorphism $\alpha_A$ to be the identity 
when $A=S(p)$ for $p\in Ob(D)$, we see that 
$V\circ S=\widetilde{T}$.
Hence, from the universal property, we should 
get that $V\circ U\simeq {\rm id}_{\cC(T)}$.
Since $V:\cE\to \cC(T)$ is faithful,
$U:\cC(T) \to \cE$ is fully faithful.
Condition \ref{prop:criterion-equi}(b) 
implies now that $U$ is also essentially 
surjective. \hfill $\square$

\smallskip

For later use, we need to refine 
the criterion given by Proposition
\ref{prop:criterion-equi}.
We keep the notation as in 
\ref{subsec-criterion}.

\begin{propose}
\label{refine-criterion}
We assume that $R$ is an integral domain 
with field of fractions $K$.
We denote $T_K:D\to \Kmod$ the representation 
defined by $T_K(-)=T(-)\otimes_R K$
and we consider the pro-finite dimensional $K$-algebra
$\End(T_K)$ (see 
\ref{subsec-artinian}).

We assume that the conditions
\ref{prop:criterion-equi}(a) and 
\ref{prop:criterion-equi}(b)
are satisfied and that 
$T(p)$ is a torsion-free $R$-module for
every $p\in Ob(D)$. 
Also, we assume the following 
variant of 
\ref{prop:criterion-equi}(c):
\begin{enumerate}

\item[(c$'$)] For every map $S(p)\to A$ in $\cE\otimes_RK$,  
$$\ker\{T_K(p)=G(S(p))\otimes_R K\to G(A)\otimes_R K\}$$ 
is a sub-$\End(T_K)$-module of $T_K(p)$.

\end{enumerate}
Then $U$ is an equivalence of categories.

\end{propose}

\begin{proof}
Using Proposition \ref{prop:criterion-equi}, it remains to show that 
\ref{refine-criterion}(c$'$) implies
\ref{prop:criterion-equi}(c).
First, note that if $G(A)$ is a torsion-free $R$-module, 
then we have
\begin{equation}
\label{eq:refine-criterion-1}
\ker\{T(p)\to G(A)\}=
T(p)\cap \ker\{T_K(p)\to G(A)\otimes_RK\}
\end{equation}
where the intersection is taken inside $T_K(p)$.
Let $E\subset D$ be a finite sub-diagram containing $p$ and 
such that 
$${\rm Im}\{\End(T_K)\to \End(T_K(p))\}=
{\rm Im}\{\End(T_K|_E)\to \End(T_K(p))\}.$$
By \ref{refine-criterion}(c$'$), 
$\ker\{T_K(p)\to G(A)\otimes_RK\}$ is an 
sub-$\End(T_K|_E)$-module of $T_K(p)$.
As $\End(T_K|_E)=\End(T|_E)\otimes_RK$, we 
see that 
the right hand side of 
\eqref{eq:refine-criterion-1}
is an sub-$\End(T|_E)$-module of 
$T(p)$. This shows that
\ref{prop:criterion-equi}(c)
is true for $S(p)\to A$. 

For the general case we argue as follows. 
Given $A$, we may find an epimorphism
$e:A'\twoheadrightarrow A$ such that 
$G(A')$ is torsion-free (\eg use \ref{prop:criterion-equi}(b)).
We then choose an epimorphism
$S(p')\twoheadrightarrow S(p)\times_A A'$
and consider the
commutative square
$$\xymatrix{S(p') \ar[r]  \ar@{->>}[d]^-f & A'\ar@{->>}[d]^-e \\
S(p) \ar[r] & A.}$$
By the previous discussion, we know 
\ref{prop:criterion-equi}(c)
for $S(p')\to A'$; 
let $E'\subset D$ be a finite sub-diagram containing $p'$ and
such that $\ker\{T(p')\to G(A')\}$ is a sub-$\End(T|_{E'})$-module
of $T(p')$. 
On the other hand, by Lemma
\ref{lem:S-p-p-prime}, 
$G(f):G\circ S(p')\to G\circ S(p)$ is a morphism of 
$\End(T|_E)$-modules with $E$ a finite sub-diagram of $D$ containing $p$ and $p'$. As
$$G(f)\left(\ker\{T(p')\to G(A')\}\right)=\ker\{T(p)\to G(A)\}$$
we see immediately that 
$\ker\{T(p)\to G(A)\}$ is a sub-$\End(T|_{E\cup E'})$-module 
of $T(p)$.
\end{proof}

\section{Nori $n$-motives}

We here define the categories $\M_n$ of homological 
mixed $n$-motives. 
We also introduce some related categories which
we denote $\M_n'$. These are obtained by
simply repeating Nori's construction
while restricting ourself to homological degree 
less or equal to $n$.

\subsection{Nori's diagram} Recall the definition 
of Nori's diagram ${}^D(Sch_k)$ from \cite{LV}.
Objects are triples $(X,Y,i)$ where $X$ is a $k$-scheme, $Y\subset X$ is a closed subset and $i$ is an integer. 
Arrows are of the following kinds:
\begin{itemize}

\item[a)] $f:(X,Y,i)\to (X',Y',i)$ for any morphism $f: X\to X'$ such that $f(Y)\subset Y'$ and

\item[b)] $\delta:(X,Y,i)\to (Y,Z,i-1)$ for any $Z\subset Y \subset X$ closed in $X$.

\end{itemize}
We have a canonical representation (associated to the complex
embedding $\sigma:k\hookrightarrow \C$)
\begin{equation}\label{can-rep-sch-k}
H_*:{}^D(Sch_k)\to \Rmod
\end{equation}
given by $(X,Y,i)\leadsto H_i(X,Y;R)$ 
the singular homology of the pair $(X^{\an},Y^{\an})$ with 
$R$-coefficients. For $f$ and $\delta$ as before, 
$H_*(f)\df f_*$ is given by the functoriality of 
singular homology and 
$H_*(\delta)\df \partial$ is the boundary map in the long 
exact sequence associated to 
the triple $(X^{\an},Y^{\an},Z^{\an})$.

\begin{defn} 
\label{Nori-mot}
With the above notation, we set
\begin{equation}
\label{eqn:Nori-mot}
\M^R\df \cC \left(H_*:{}^D(Sch_k)\to \Rmod\right).
\end{equation}
This is the $R$-linear 
category of {\em Nori effective motives}.
When $R=\Z$, we simply write $\M$. 
If we wish to stress the dependence on the field $k$, we write
$\M^R(k)$ and $\M(k)$.

Given a triple $(X,Y,i)\in {}^D(Sch_k)$, we
denote $\widetilde{H}_i(X,Y;R)$ its image in 
$\M^R$  by the universal representation 
of ${}^D(Sch_k)$ associated to 
$H_*$. 
\end{defn}

\begin{defn}\label{defn:nori-n-mot-thick}
We denote $\M_n^R$ (or $\M_n^R(k)$ if we wish to stress the dependence on $k$) the thick abelian subcategory of $\M^R$ 
generated by $\widetilde{H}_i(X,Y;R)$
where $X$ is a $k$-scheme, $Y\subset X$ is a closed subset 
and $i$ is an integer such that $i\leq n$.
Objets in $\M_n^R$ are called \emph{Nori $n$-motives}.
When $R=\Z$, we simply write $\M_n$ (or $\M_n(k)$). 
\end{defn}

\begin{propose}\label{generate} 
In \ref{defn:nori-n-mot-thick} we can restrict to $k$-schemes
of dimension at most $n$, \ie 
$\M_n^R$ is the thick abelian subcategory of 
$\M^R$ generated by 
$\widetilde{H}_i(X,Y;R)$
where $X$ is a $k$-scheme of dimension at most $n$, 
$Y\subset X$ is a closed subset 
and $i$ is an integer such that $i\leq n$.
\end{propose}

\begin{proof}
Let $\mathcal{A}\subset\M_n^R$
be the thick abelian subcategory generated by the objects 
as in the statement. 
It is enough to show that 
$\widetilde{H}_i(X,Y;R)\in \mathcal{A}$
for $X$ a $k$-scheme of arbitrary dimension, 
$Y\subset X$ a closed subset 
and $i\leq n$. 
We derive this by a standard argument relying on 
Lefschetz theorem (on hyperplane sections). We reproduce
this argument 
for the sake of completeness. 

We argue by induction on the dimension of $X$.
We may assume that ${\rm dim}(X)\geq n+1$; otherwise there is nothing to
prove. 
We may also assume that ${\rm dim}(Y)<{\rm dim}(X)$. 
(Indeed, replacing $X$ and $Y$ by the closures of the complements 
of the common irreducible components does not change the 
relative homology.) If $Y'\subset X$ is a closed subset containing 
$Y$, we have 
a short exact sequence
$$\widetilde{H}_i(Y',Y;R) \to \widetilde{H}_i(X,Y;R)\to 
\widetilde{H}_i(X,Y';R).$$
If ${\rm dim}(Y')<{\rm dim}(X)$, 
it suffices by induction to 
treat the case of $\widetilde{H}_i(X,Y';R)$. 
In other words, we may enlarge $Y$ and assume that 
$X\smallsetminus Y$ is smooth and that there exists a blow-up 
with center contained in $Y$ rendering $X$ a quasi-projective scheme.
(This is possible by Chow's lemma.)

Next, given a blow-up $X_1\to X$ inducing an isomorphism
$X_1\smallsetminus Y_1\to X\smallsetminus Y$, with $Y_1=X_1\times_XY$, 
we have an isomorphism
$\widetilde{H}_i(X_1,Y_1;R)\simeq 
\widetilde{H}_i(X,Y;R)$. Thus, using Chow's lemma and Hironoka's
resolution of singularities, we
may assume that $X$ is smooth and quasi-projective. Moreover, we may fix an embedding 
$X\hookrightarrow \bar{X}$ into a smooth and projective $k$-scheme $\bar{X}$ such that 
$X_{\infty}=\bar{X}\smallsetminus X$ is a simple normal crossing divisor.

Now, using the exact sequence
$$\tilde{H}_i(X,\emptyset;R) \to 
\tilde{H}_i(X,Y;R)\to \tilde{H}_{i-1}(Y,\emptyset;R)$$
and induction, it is enough to show that 
$\tilde{H}_i(X,\emptyset;R)\in \mathcal{A}$.
Let $\bar{Z}\subset \bar{X}$ be a smooth ample divisor
meeting transversally the divisor 
$X_{\infty}$. Letting $Z=X\cap \bar{Z}$, 
we claim that 
the map 
$$\tilde{H}_i(Z,\emptyset;R)\to \tilde{H}_i(X,\emptyset;R)$$
is surjective if $i={\rm dim}(X)-1$ (resp. bijective if
$i<{\rm dim}(X)-1$); this will finish the proof.
It is enough to check this after applying the 
forgetful functor $\M\to \Rmod$. 
We argue by induction on the dimension of $X$ and the number of 
irreducible components in $X_{\infty}$. If
$X_{\infty}=\emptyset$, then the claim is simply 
Lefschetz hyperplane theorem.
Otherwise, let $\overline{D}\subset X_{\infty}$ be an irreducible 
component, $E\subset X_{\infty}$ 
the union of the remaining irreducible components and $D=\overline{D}\smallsetminus E$.
We then have a commutative diagram
$$\xymatrix@R=1.7pc@C=1.5pc{H_{i-1}(D;R) \ar[r] \ar[d] & H_i(X;R) \ar[r] \ar[d] & H_i(\overline{X}\smallsetminus E;R) \ar[r] \ar[d] & H_{i-2}(D;R) \ar[d]^-{\sim}\\
H_{i-1}(D\cap \overline{Z};R) \ar[r] & H_i(Z;R) \ar[r] & H_i(\overline{Z}\smallsetminus E;R) \ar[r] & H_{i-2}(D\cap \overline{Z};R)}$$
where the lines are part of Gysin long exact sequences.
The first and third vertical arrows are surjective 
(resp. bijective) by induction. 
The surjectivity (resp. bijectivity) of the second vertical 
arrow follows by a simple diagram chasing. 
\end{proof}

\subsection{\/}
Consider the full sub-diagram 
${}^D(Sch_k)_{\leq n}\subset {}^D(Sch_k)$
consisting of triples $(X,Y,i)$ with $i\leq n$.
As in \ref{Nori-mot} we set
\begin{equation}
\label{eqn:Nori-n-mot}
\M'^R_n\df \cC \left(H_*:{}^D(Sch_k)_{\leq n}
\to \Rmod\right).
\end{equation}
(When $R=\Z$, we simply write $\M'_n$.
If we wish to stress the dependence on the field $k$, we write
$\M'^R_n(k)$ and $\M'_n(k)$.)
For $i\leq n$, we also denote 
$\widetilde{H}'_i(X,Y;R)$ the object in 
$\M_n'^R$ associated to 
$(X,Y,i)$ by the universal representation
$$\widetilde{H}'_*:{}^D(Sch_k)_{\leq n}
\to \M_n'^R.$$
By universality, there is a faithful exact functor
\begin{equation}
\label{eq:compar-n-mot}
\M'^R_n\to \M^R_n.
\end{equation}
We conjecture that 
\eqref{eq:compar-n-mot} is an equivalence of categories although
we expect this to be a difficult problem.

\begin{lemma}
\label{lim} 
Let $n\leq n'$ be non negative integers. There are (faithful, exact) functors $\M'^R_n\to \M'^R_{n'}$ inducing an equivalence  
$$2-\colim{n \geq 0} \M'^R_n \cong \M^R.$$
\end{lemma}

\begin{proof}
The first claim in clear and 
the second 
is a particular case of \eqref{colim-sim-C(T)}. 
\end{proof}

\begin{propose}\label{generate-prime}
The category $\M'^R_n$ is generated, as a thick 
abelian subcategory of itself, by 
$\widetilde{H}'_i(X,Y;R)$ 
where $X$ is a $k$-scheme of dimension at most $n$, 
$Y\subset X$ is a closed subset 
and $i$ is an integer such that $i\leq n$.
\end{propose}

\begin{proof}
A proof is obtained by replacing $\widetilde{H}_i$
by $\widetilde{H}'_i$ everywhere in the proof 
of Proposition \ref{generate}. 
\end{proof}

\section{The easy case of Artin motives}

In this section, we consider the case 
$n=0$.

\subsection{\/}\label{0-motive-categ}
Let $\cM_0^R=\cM_0^R(k)$ be the category of $0$-motives 
(with coefficients in $R$), 
\ie $R$-constructible \'etale sheaves on $Et_k$, 
the small \'etale site of $k$. (Recall that a sheaf 
on $Et_k$ is $R$-constructible if and only if 
it is locally constant and its stalks are 
finitely generated $R$-modules.) 
Clearly, $\cM_0^R$
is abelian. Moreover, the fiber functor associated 
to the geometric point given by the 
complex embedding $\iota:k\hookrightarrow \C$, 
yields a forgetful functor 
\begin{equation}
\label{fiber-funct}
\iota^*:\cM_0^R \to \Rmod.
\end{equation}
Note the following observation.

\begin{lemma}
\label{h-M0}
There exist a representation 
$S:{}^D(Sch_k)_{\leq 0} \to \mathcal{M}_0^R$
and an isomorphism of representations
$\iota^*\circ S\simeq H_*$.
\end{lemma}

\begin{proof}
An objet of ${}^D(Sch_k)_{\leq 0}$
consists of a triple $(X,Y,0)$. 
We define
$$S(X,Y,0)\df 
\frac{\pi_0(X)\otimes R}{\pi_0(Y)\otimes R}.$$
(In the above formula, $\pi_0(Z)$ denote the 
\'etale $k$-scheme of geometric connected components of 
a $k$-scheme $Z$ and, for $V\in Et_k$,  
$V\otimes R$ is the \'etale sheaf associated
to the presheaf: $U\in Et_k \leadsto \bigoplus_{U\to V} R$.)
The verification that $S$ gives a representation is 
easy and is left to the reader. 
Almost from the construction we have 
$$\iota^*\circ S(X,Y,0)=\frac{\bigoplus_{x\in \pi_0(X^{\rm an})} R}{\bigoplus_{y\in \pi_0(Y^{\rm an})}R}.$$
This gives the required identification 
$\iota^*\circ S=H_*$.
\end{proof}

\subsection{\/} From Lemma
\ref{h-M0} and universality (\ie 
\cite[Theorem 41]{LV}) we get a canonical 
functor:
\begin{equation}
\label{MN0-to-M0}
\M'^R_0\to \mathcal{M}_0^R.
\end{equation}

\begin{propose} 
\label{N0} 
The functor \eqref{MN0-to-M0}
is an equivalence of categories.
\end{propose}

\begin{proof}
We will use Proposition
\ref{prop:criterion-equi}.
Only the condition \ref{prop:criterion-equi}(c) needs to be discussed.

So let $\cF$ be an $R$-constructible 
\'etale sheaf on $Et_k$ and let 
$$\alpha:\frac{\pi_0(X)\otimes R}{\pi_0(Y)\otimes R}\to
\cF$$
be a morphism.
Let $l/k$ be a finite Galois extension 
which trivializes $\pi_0(X)$, $\pi_0(Y)$ and $\cF$.
Let $E$ be the full sudiagram 
whose objects are $(\Spec(l),\emptyset,0)$
and $(X,Y,0)$.
Then it is easy to see that 
$\End(T|_E)=R[{\rm Gal}(l/k)]$, the group algebra of the Galois group.
The map $\iota^*(\alpha)$ being ${\rm Gal}(l/k)$-equivariant, it 
follows that $\ker(\iota^*(\alpha))$ is a sub-$\End(T|_E)$-module. 
(A similar argument appears in \cite[\S 6.1]{NN}.)
\end{proof}

\section{Deligne $1$-motives as a universal category} 
In this section, we concentrate on the category 
$\M''_1$ which we will define below.  
Our main goal is to identify it with the abelian
category of $1$-motives with 
torsion ${}^t\cM_1$ (see \cite[\S1]{BRS} and \cite[Appendix C]{BK}).

\subsection{A smaller diagram of curves}
Let
${}^D(Crv_k)$
be the full sub-diagram of ${}^D(Sch_k)$ consisting of triples
$(C,Z,1)$ where $C$ is a smooth affine curve and 
$Z\subset C$ a closed subset consisting of finitely many closed 
points.

By restriction, we get a representation 
$H_*:{}^D(Crv_k)
\to \Rmod$. We set 
$$\M_1''^R=\mathcal{C}\left(H_*:{}^D(Crv_k)
\to \Rmod\right).$$
(As usual, when $R=\Z$, we simply write
$\M''_1$. If we want to stress the dependence on the field $k$, we 
write $\M_1''^R(k)$ and $\M_1''(k)$.)
The image in 
$\M_1''^R$ of an object $(C,Z,1)\in {}^D(Crv_k)$ 
is denoted by $\widetilde{H}''_1(C,Z;R)$.
By universality, we get  
faithful and exact functors 
\begin{equation}
\label{compare-M0s}
\M_1''^R \to \M_1'^R \to \M_1^R.
\end{equation}

\subsection{\/} Let ${}^t\cM_1={}^t\cM_1(k)$
be the abelian category of $1$-motives
with torsion \cite[\S1]{BRS}.
Recall that a $1$-motive with torsion is a complex of 
commutative group schemes
$[\cF\to \cG]$ where 
$\cG$ is a semi-abelian variety and
$\cF$ is a lattice with torsion, \ie 
a $\Z$-constructible \'etale sheaf on $Et_k$ (as in 
\ref{0-motive-categ}) considered as a group 
scheme in the obvious way.

Given an \'etale $k$-scheme $V$, we denote 
$\Z_{tr}(V)$ the group scheme 
such that, for every $k$-scheme $X$, 
$\Z_{tr}(V)(X)$ is the free abelian group on the set
of connected components of $X\times_kV$.
This corresponds to the $\Z$-constructible 
\'etale sheaf $V\otimes \Z$ used in the proof of Lemma 
\ref{h-M0}.

\subsection{\/}
We have a representation 
$$A:{}^D(Crv_k)\to
{}^t\cM_1$$
given by 
$$(C,Z, 1)\leadsto A(C,Z)\df[{\rm Div}^0_Z(C) \to \Alb^0(C)]$$
where ${\rm Div}^0_Z(C)$ is the lattice $\ker\{\Z_{tr}(Z)
\to \Z_{tr}(\pi_0(C))\}$ and $\Alb^0(C)$ is the connected component of the identity of the Serre-Albanese scheme $\Alb (C)$ of $C$. 
(Note that $A(C,Z) = \Alb^{-}(C, Z)$ with the notation adopted in \cite{BS}
and $A(C,Z) \cong\LA{1} (C,Z)$ according to \cite{BK}.)
On the other hand, we have a functor 
\begin{equation}
\label{Betti1}
T_{\Z}:{}^t\cM_1\to \Zmod
\end{equation}
given as follows:  
$$T_{\Z}([\cF\to \cG])=\cF(\C)\times_{\cG(\C)} {\rm Lie}\,\cG(\C)$$
This is a forgetful functor. Moreover, we have:

\begin{lemma}
\label{HsimAcircB}
There is a canonical isomorphism
$H_*\simeq T_{\Z}\circ A$ between representations of
${}^D(Crv_k)$.
\end{lemma}

\begin{proof}
This is classical: see \cite[Proposition 3.1.2 \& \S 5.3]{BS}. 
For the sake of completeness, we recall the description of the isomorphism for $(C,Z,1)\in {}^D(Crv_k)$ 
assuming $k=\C$, $Z$ non empty, and
$C$ connected and not isomorphic to the affine line. 
The last assumption is to insure that there is an 
embedding $C\hookrightarrow \Alb^1(C)$ where 
$\Alb^1(C)\subset \Alb(C)$ is the connected component consisting of $0$-cycles of degree $1$. The group $H_1(C,Z;\Z)$ is generated by 
classes of paths $\gamma:[0,1]\to C^{\an}$ such that $\gamma(0), \, \gamma(1)\in Z$.
Let $\gamma^1:[0,1]\to \Alb^1(C)^{\an}$ be the composition of $\gamma$ 
with 
the canonical embedding. Fix a universal cover 
$U\to \Alb^1(C)^{\an}$; then $U$ is naturally a torsor over the 
vector space 
$V=\Lie \Alb^0(C)$. If $\tilde{\gamma}:[0,1]\to U$ is any lift of 
$\gamma^1$, then $\tilde{\gamma}(1)-\tilde{\gamma}(0)\in V$ is
well-defined. Moreover, it is an element of 
$Z\times_{\Alb^0(C)^{\rm an}}V\subset T_{\Z}(A(C,Z))$.
\end{proof}

\subsection{\/}
By Lemma \ref{HsimAcircB} and universality
(\ie \cite[Theorem 41]{LV})
there is a canonical, faithful and
exact functor 
\begin{equation}
\label{main-fonct-M1-cM1}
\M_1''\to {}^t\cM_1.
\end{equation}
We can now state one of the key results of this paper.

\begin{thm}
\label{thm:main-thm}
The functor 
\eqref{main-fonct-M1-cM1}
is an equivalence of categories.
\end{thm}

\subsection{\/}

From now on, we only use the coefficients rings 
$R=\Z$ or $R=\Q$.
As a particular case of 
\eqref{chg-R-from-Z-to-Q}, we have 
$\M''^{\,\Q}_n=\M''_n\otimes \Q$.
We denote $\mathcal{M}_1^{\Q}={}^t\mathcal{M}_1\otimes \Q$,
the abelian category of $1$-motives
where isogenies are inverted; we have 
a forgetful functor 
$T_{\Q}:\mathcal{M}_1^{\Q}\to \Qmod$ 
induced from
\eqref{Betti1}.

Let $\cR_1$ be the pro-finite dimensional $\Q$-algebra of 
endomorphisms of 
the representation $H_*:{}^D(Crv_k) \to \Qmod$
(see \ref{subsec-artinian}).

\begin{propose}
\label{key-propose}
Let $(C,Z,1)$ be in ${}^D(Crv_k)$ and 
$\alpha:A(C,Z) \to M$ a morphism in 
$\mathcal{M}_1^{\Q}$. Then, the kernel of 
$$H_1(C,Z;\Q)\simeq T_{\Q}\circ A(C,Z) 
\longby{T_{\Q}(\alpha)}T_{\Q}(M)$$
is a sub-$\cR_1$-module of 
$\tilde{H}''_1(C,Z;\Q)$.
\end{propose}
\begin{proof} The proof is contained in \ref{proof-key-prop} below and it relies on Lemmas \ref{lemma-key-prop-1.5}, \ref{lemma-key-prop-2}, 
\ref{lemma-key-prop-4}, \ref{lemma-key-prop-5}, \ref{lemma-key-prop-6} and \ref{lemma-key-prop-7} below. 
\end{proof}

\subsection{\/} 
Once Proposition \ref{key-propose} is proven, 
Theorem \ref{thm:main-thm} will follow from 
Proposition \ref{refine-criterion}.
We break the proof of Proposition \ref{key-propose} into small steps. 
In what follows, the kernel in the statement of Proposition \ref{key-propose}
will be denoted 
$K(C,Z,\alpha)$; it is a subspace of $H_1(C,Z;\Q)$.

\begin{lemma}
\label{lemma-key-prop-1.5}
Assume that 
$M=N\oplus N'$ (in $\mathcal{M}_1^{\Q}$) and let $\beta$ and 
$\beta'$ be the corresponding components of 
$\alpha$. If 
\emph{Proposition \ref{key-propose}}
holds for $\beta$ and 
$\beta'$ then it holds for $\alpha$.
\end{lemma}

\begin{proof}
Indeed, we have
$K(C,Z,\alpha)=
K(C,Z,\beta)\cap K(C,Z,\beta')$.
\end{proof}

\begin{lemma}
\label{lemma-key-prop-2}
Let $f:(D,T,1) \to (C,Z,1)$ be an arrow in ${}^D(Crv_k)$
such that $f:D \to C$ is dominant and the image of 
the induced morphism 
$$f_*:{\rm Div}_T^0(D) \to {\rm Div}_Z^0(C)$$
has a finite index. 
Then 
\emph{Proposition \ref{key-propose}}
for $\alpha:A(C,Z)\to M$ follows from the case of 
$\alpha \circ A(f):A(D,T) \to M$.
\end{lemma}

\begin{proof}
Indeed, the hypothesis of the lemma imply that the homomorphism 
$H_1(D,T;\Q) \to H_1(C,Z,\Q)$ is surjective. 
If follows that $K(C,Z,\alpha)$ is the image 
of $K(D,T,\alpha')$ with 
$\alpha'=\alpha \circ A(f)$.
This finishes the proof.
\end{proof}

\begin{remark}
\label{rem-key-prop-3}
Assume that $f:D\to C$ induces a bijection 
on the sets of connected components and that 
$f(T)=Z$.
Then the hypothesis of Lemma
\ref{lemma-key-prop-2}
are satisfied. 
\end{remark}

\begin{lemma}
\label{lemma-key-prop-4}
\emph{Proposition \ref{key-propose}} holds when 
$M$ is a $0$-motive, \ie 
$M=[\mathcal{F}\to 0]$.

\end{lemma}

\begin{proof}
Using Lemma \ref{lemma-key-prop-1.5},
we may assume that the lattice
$\mathcal{F}$ is simple (as an object of 
$\mathcal{M}_0^{\Q}$). Therefore, there exists a 
finite extension 
$l/k$ such that $\mathcal{F}$ is a direct factor of 
$\Z_{tr}(l)$ (in $\mathcal{M}_0^{\Q}$). 
Thus, we may assume that 
$\mathcal{F}=\Z_{tr}(l)$. We may also enlarge 
$l$ and assume that $l$ is Galois and contains the residue field 
of every point in $Z$. Let $C'=C\otimes_k l$ and $Z'=Z\otimes_k l$. 
By Lemma 
\ref{lemma-key-prop-2}, we may replace 
$(C,Z)$ by $(C',Z')$. In other words, we may assume that 
$C$ is defined over $l$ and every point of $Z$ is rational over $l$.

Now, ${\rm Div}^0_Z(C)$ is the kernel of 
$\Z_{tr}(Z) \to \Z_{tr}(\pi_0(C))$.
Hence, it is a direct factor of $\Q_{tr}(Z)$. 
On the other hand, $\Hom (\Z_{tr}(Z),\Z_{tr}(l))$ has a basis 
which is indexed by $(z,\sigma)$ where 
$z$ is a point of $Z$ and 
$\sigma:k(z)\simeq l$ is a $k$-isomorphism. A 
couple $(z,\sigma)$ corresponds to the composition of
$$u(z,\sigma):\Z_{tr}(Z) \to \Z_{tr}(k(z)) \overset{\sigma}{\simeq} \Z_{tr}(l).$$
It follows that 
$\alpha_0:{\rm Div}^0_Z(C) \to \Z_{tr}(l)$ 
can be written as the composition of 
${\rm Div}^0_Z(C) \hookrightarrow \Z_{tr}(Z)$ 
and a linear combination 
$\sum_{(z,\sigma)}a_{z,\sigma}\cdot u(z,\sigma)$
with $a_{z,\sigma}\in \Q$. (Recall that 
$\alpha$ is a morphism in $\mathcal{M}_1^{\Q}$.)

We claim that, for a fixed $(z,\sigma)$, the composition of
\begin{equation}
\label{eq-lem-key-prop-4}
H_1(C,Z;\Q)\simeq T_{\Q}A(C,Z) \!\xymatrix@C=1.1pc{\ar[r]^-{u(z,\sigma)} &} \! T_{\Q}(\Z_{tr}(l))\simeq H_0(\Spec(l);\Q)
\end{equation}
coincides with a morphism obtained from a zigzag in the diagram ${}^D(Crv_k)$ modulo the isomorphism $H_0(\Spec(l);\Q)\simeq 
H_1(\Aff^1_{\ell}, \{0,1\};\Q)$. To show this, we need a construction.
Let $z_0\in Z$ be a point different from $z$. (We may assume that 
$Z\neq \{z\}$ because otherwise, the morphism 
\eqref{eq-lem-key-prop-4} is necessarily zero.)
Consider a finite morphism
$C\to \Aff^1_l$ which is injective on $Z$ and sends
$z_0$ to the zero section. Denote 
$T\subset \Aff^1_l$ the image of 
$Z$ and $t\in T$ the image of $z$. Then our zigzag is the following:
$$(C,Z,1) \rightarrow (\Aff^1_l,T,1) \leftarrow
(\Aff^1_l,\{0,t\},1)\sqcup (\Aff^1_l,T\smallsetminus \{t\},1)\overset{(*)}{\rightarrow}$$
$$(\Aff^1_l,\{0,t\},1)\sqcup (\Aff^1_l,\{0\},1) \leftarrow 
(\Aff^1_l,\{0,t\},1)\overset{t^{-1}}{\to} (\Aff^1_l,\{0,1\},1)$$
where the arrow $(*)$ is given by the identity on the first factor and
by the zero morphism on the second factor.
It follows that 
\eqref{eq-lem-key-prop-4}
is a morphism of $\cR_1$-modules.
This proves that 
$$H_1(C,Z;\Q)\simeq T_{\Q}A(C,Z) \!\xymatrix@C=1.1pc{\ar[r]^-{\alpha} &} \! T_{\Q}(\Z_{tr}(l))\simeq H_0(\Spec(l);\Q)$$
is also a morphism of $\cR_1$-modules. Hence, its kernel is a sub-$\cR_1$-module of $H_1(C,Z;\Q)$.
\end{proof}

\begin{lemma}
\label{lemma-key-prop-5}
Let $(C,Z,1)$ be an object of ${}^D(Crv_k)$ and let
$$\beta:[\mathcal{L}\to 0]\to A(C,Z)$$ 
be a morphism in $\mathcal{M}_1^{\Q}$
from a lattice $\mathcal{L}$.
Then, the image of the composition
$$T_{\Q}\mathcal{L}\to T_{\Q}A(C,Z)\simeq H_1(C,Z;\Q)$$
is a sub-$\cR_1$-module of $H_1(C,Z;\Q)$.  
\end{lemma}

\begin{proof}
It suffices to consider the case where $\mathcal{L}$ is simple
in $\mathcal{M}_0^{\Q}$. 
In this case, $\mathcal{L}$ is a direct factor of $\Z_{tr}(l)$
(in $\mathcal{M}_0^{\Q}$). 
Thus, we may assume that $\mathcal{L}=\Z_{tr}(l)$. 
Then a multiple of $\beta$ corresponds to a morphism
of lattices 
$$\beta_0:\Z_{tr}(l)\to \Z_{tr}(Z)$$ 
whose image is contained in ${\rm Div}^0_Z(C)$
and such that the composition
$$\Z_{tr}(l) \to {\rm Div}^0_Z(C) \to \Alb^0(C)$$
is zero. 
Therefore we can find 
a finite correspondence
$\gamma:\Aff^1_l \to C$
such that $\gamma \circ (i_1-i_0)=j_Z\circ \beta_0$
where $j_Z:Z\hookrightarrow C$ is the inclusion.

Recall that we want to show that the image of 
$$T_{\Q}(\Z_{tr}(l)) \overset{\beta}{\to} T_{\Q}A(C,Z) \simeq 
H_1(C,Z;\Q)$$ 
is a sub-$\cR_1$-module. 
For this, we are free to add to $Z$ any closed subset 
$T\subset C\setminus Z$ of dimension $0$.
We may find a finite Galois cover 
$c:E\to \Aff^1_l$ such that
$\gamma\circ c$ is a linear combination of maps 
from $E$ to $C$, i.e,
$\gamma\circ c=\sum_{i=1}^n a_i f_i$ with $a_i\in \Z$. 
We take $T=\bigcup_{i=1}^n f_i(F) \setminus Z$
with $F=c^{-1}\{0,1\}$.

Now, let $G$ be the Galois group of $c:E \to \Aff^1_l$.
Then $G$ acts on the $\cR_1$-module 
$\widetilde{H}''_1(E,F;\Q)$ and 
the canonical map
$$H_1(E,F;\Q)\to H_1(\Aff^1,\{0,1\};\Q)$$ 
identifies 
$H_1(\Aff^1,\{0,1\};\Q)$
with the sub-$\cR_1$-modules of invariants. 
On the other hand, the maps
$$f_{i*}:H_1(E,F;\Q) \to H_1(C,Z\cup T;\Q)$$
are $\cR_1$-linear. 
Hence, the image of 
$H_1(E,F;\Q)^{G}$ by $\sum_{i=1}^n a_i f_{i*}$ is a 
sub-$\cR_1$-module of 
$\widetilde{H}''_1(C,Z\cup T;\Q)$.
By construction, it coincides with the image of the composition of
$$T_{\Q}(\Z_{tr}(l)) \overset{\beta}{\to} T_{\Q}A(C,Z\cup T) \simeq 
H_1(C,Z\cup T;\Q).$$ 
This finishes the proof of the lemma.
\end{proof}

We still need some more lemmas.

\begin{lemma}
\label{lemma-key-prop-6}
It suffices to prove \emph{Proposition
\ref{key-propose}} when $M=[\mathcal{F} \to \Alb^0(D)]$ 
with $D$ a smooth affine curve.

\end{lemma}

\begin{proof}
Let $M=[\mathcal{F} \to \mathcal{G}]$ be a general $1$-motive.
It is enough to show that 
$\mathcal{G}$ embeds, up to isogeny, into $\Alb^0(D)$ with 
$D$ a smooth and affine curve. 
Indeed, letting  
$N=[\mathcal{F} \to \Alb^0(D)]$, one gets a monomorphism (in 
$\mathcal{M}_1^{\Q}$) of $1$-motives
$u:M\to N$ and the equality 
$K(C,Z,\alpha)=K(C,Z,u\circ \alpha)$ holds.

We now construct a monomorphism, up to isogeny, 
of group schemes
$\mathcal{G} \hookrightarrow \Alb^0(D)$.
If there is a an isogeny between
$\mathcal{G}$ and a product 
$\mathcal{G}_1\times \mathcal{G}_2$, 
it is enough to consider
$\mathcal{G}_1$ and $\mathcal{G}_2$ separately.
(If $\mathcal{G}_i\subset \Alb^0(D_i)$ for $i\in \{1,2\}$, 
just take $D=D_1\sqcup D_2$.)
If $\mathcal{G}$ is a torus, which splits over 
a finite extension $l/k$, one can embed 
$\mathcal{G}$ into a product of tori of the form
$\G_m\otimes\Z_{tr}(l)=\Alb^0(\Aff^1_l\smallsetminus\{0\})$.
Therefore, it remains to treat the case where $\mathcal{G}$ is not isogenous to a product where one of the factor is a non trivial torus.
This is equivalent to say that $\mathcal{G}$ has non nontrivial map 
to a torus. 
Consider the dual $1$-motive $\mathcal{G}^{\vee}$. It is of the form $[\mathcal{L} \to \mathcal{A}]$ with $\mathcal{A}$ an abelian variety
and $\mathcal{L}$ a torsion-free lattice.
Our condition on $\mathcal{G}$ implies that the map
$\mathcal{L}\to \mathcal{G}$ is a monomorphism of schemes. 
By drawing a general smooth curve $\overline{D}\subset \mathcal{A}$ containing generators of $\mathcal{L}$, we obtain a pair $(\overline{D},T)$ with a surjective morphism
$A(\overline{D},T)\twoheadrightarrow \mathcal{G}^{\vee}$. Dualizing back, we see that $\mathcal{G}$ injects inside 
$\Alb^0(D)$ with $D=\overline{D}-T$.
\end{proof}

From now on, we assume that 
$M=[\mathcal{F}\to \Alb^0(D)]$
with $D$ a smooth affine curve. 
A multiple of the morphism 
$\alpha:A(C,Z) \to M$ induces a morphism of 
semi-abelian varieties
$\alpha_1:\Alb^0(C) \to \Alb^0(D)$.

\begin{lemma}
\label{lemma-key-prop-7}
To prove \emph{Proposition
\ref{key-propose}},
we may assume that 
$\alpha_1$ is induced 
by a linear combination of maps from $C$ to $D$. 
In other words, we may assume that 
$\alpha_1=\sum_{i=1}^n a_i \Alb^0(f_i)$ for some maps 
$f_i:C \to D$ and integers $a_i\in \Z$.
\end{lemma}

\begin{proof}
Let $l/k$ be a finite Galois extension with Galois group
$G$. Consider the 
morphism of $1$-motives 
$$\alpha'=\alpha\otimes \Z_{tr}(l):A(C\otimes_kl,Z\otimes_kl)\simeq A(C,Z)\otimes \Z_{tr}(l) 
\to M\otimes \Z_{tr}(l).$$
Clearly, this is a $G$-equivariant morphism. It follows that 
$G$ acts on $K(C\otimes_kl,Z\otimes_kl,\alpha')$ and that
the space of $G$-invariants identifies 
with $K(C,Z,\alpha)$. 
Thus, it is enough to prove 
Proposition
\ref{key-propose} for $\alpha'$. Using this, we may assume that 
both projections 
$C\to \pi_0(C)$ and $D\to \pi_0(D)$ have sections.

The section of $C\to \pi_0(C)$ is used to construct a
retraction 
$r:\Alb(C)\to \Alb^0(C)$ to the natural inclusion. 
(This is needed in \eqref{psi-below}
below.)

Let $h_0(D)$ the homotopy invariant preasheaf with transfers
(on the category $Sm_k$ of smooth $k$-schemes) 
associated to $\Z_{tr}(D)$.
There is an obvious morphism 
of homotopy invariant presheaves with transfers
\begin{equation}
\label{eq:Pic-R-vs-S-A}
h_0(D)\to \Alb(D).
\end{equation}
The section 
$D\to \pi_0(D)$ is used to ensure that 
\eqref{eq:Pic-R-vs-S-A}
induces an isomorphism on finitely generated extensions $K/k$. 
Indeed, to check this property, we may assume without loss of generality 
that $k=K$, \ie it is enough to check that 
$h_0(D)(k)\to \Alb(D)(k)$ is an isomorphism.
The group
$h_0(D)(k)$ is canonically isomorphic 
to the relative Picard group 
$Pic(\overline{D},D_{\infty})$
where $\overline{D}$ is a smooth compactification of $D$ and 
$D_{\infty}=\overline{D}\smallsetminus D$. 
Using the exact sequences
$$0\to \frac{\mathcal{O}^{\times}(D_{\infty})}{\mathcal{O}^{\times}(\pi_0(D))} \to Pic(\overline{D},D_{\infty}) \to Pic(\overline{D}) \to 0$$
and 
$$0\to \frac{\mathcal{O}^{\times}(D_{\infty})}{\mathcal{O}^{\times}(\pi_0(D))} \to \Alb(D)(k) \to \Alb(\overline{D})(k)$$
it is enough to show that 
$Pic(\overline{D})\to \Alb(\overline{D})(k)$ is surjective. 
Clearly, $\Alb(\overline{D})$ is equal to 
the Picard variety $\Pic(\overline{D})$ of $\overline{D}$. 
Moreover, we have the well-known exact sequence
(see for example \cite[p.~203]{BLR}):
$$Pic(\overline{D})\to \Pic(\overline{D})(k) \to 
H^2_{\et}(\pi_0(\overline{D}),\G_m)\to H^2_{\et}(\overline{D},\G_m).$$
The existence of a section to $D\to \pi_0(D)$
implies that the last morphism is injective. 
Thus, the first morphism is surjective as needed.

Applying \cite[Proposition 11.1]{MVW} 
to the kernel and cokernel of 
\eqref{eq:Pic-R-vs-S-A} and using the previous discussion, 
we deduce 
an isomorphism 
\begin{equation}
\label{eq:Pic-R-vs-S-A-2}
h_0(D)(C_Z)\simeq \Alb(D)(C_Z)
\end{equation}
where $C_Z$ is the spectrum of the semi-local ring 
of $C$ at the points of $Z$.
Now, the left hand side in 
\eqref{eq:Pic-R-vs-S-A-2}
is the group of finite correspondences 
from $C_Z$ to $D$ up to homotopy. 
Taking the inverse image of the element 
$\psi\in \Alb(D)(C_Z)$ given by the composition
\begin{equation}
\label{psi-below}
\psi:C_Z\hookrightarrow C\to \Alb(C) \overset{r}{\twoheadrightarrow} \Alb^0(C) \overset{\alpha_1\,}{\to} \Alb^0(D)
\hookrightarrow \Alb(D)
\end{equation}
we arrive at the following conclusion.
There exists a dense open neighborhood $C'$ of $Z$ in $C$ and a finite correspondence 
$\gamma\in Cor(C',D)$ such that the 
following diagram commutes
$$\xymatrix{\Alb^0(C') \ar[r] \ar@/_/[dr]_-{\Alb^0(\gamma)} & \Alb^0(C) \ar[d]^-{\alpha_1}\\
& \Alb^0(D).\!}$$
By Lemma \ref{lemma-key-prop-2}, we may replace $C$ by $C'$. In other words, we may assume that $\alpha_1$ itself is induced by a correspondence 
$\gamma\in Cor(C,D)$.

To finish the proof, we choose a finite cover $r:C'' \to C$ such that 
$\gamma\circ r$ is a linear combination of morphisms. Using Lemma
\ref{lemma-key-prop-2}, we may replace $C$ by $C''$ and $Z$ by $r^{-1}(Z)$. In particular, we may indeed assume that
$\gamma=\sum_{i=1}^n a_i f_i$ where $a_i\in \Z$ and
$f_i:C \to D$.
\end{proof}

\subsection{Proof of Proposition \ref{key-propose}}\label{proof-key-prop} We are now ready to complete the proof of Proposition \ref{key-propose}. 
First, remark that we may assume that 
$\mathcal{F} \to \Alb^0(D)$ is injective. Indeed,
if $\mathcal{N}$ is the kernel of this morphism and $\mathcal{I}$ its image, there is a (non canonical) decomposition 
$$M=[\mathcal{N}\to 0]\oplus [\mathcal{I}\hookrightarrow \Alb^0(D)]$$
in $\mathcal{M}_1^{\Q}$. We then apply 
Lemmas
\ref{lemma-key-prop-1.5}
and \ref{lemma-key-prop-4} (for the $0$-motive $\cN$)
to conclude.

Arguing as in the beginning of the proof of 
Lemma 
\ref{lemma-key-prop-6}, 
we may replace $(C,Z,1)$ by $(C\otimes_kl,Z\otimes_kl,1)$ and 
$M$ by $M\otimes\Z_{tr}(l)$ for any 
finite Galois extension $l/k$. 
Therefore, we may assume that there exists 
such $l/k$, with Galois group $G$, such that the following properties
are satisfied:
\begin{itemize}

\item $\mathcal{F}=\bigoplus_{s=1}^r \Z_{tr}(l) e_s$
where $e_s\in \mathcal{F}(l)$ form a basis of the 
$\Z[G]$-module $\mathcal{F}(l)$.

\item The image of $e_s$ in $\Alb^0(D)(l)$ is represented
by a $0$-cycle $g_s$ of $D\otimes_kl$.

\end{itemize}

Let $T\subset D$ be a finite set of closed points
containing the supports of the $0$-cycles $g_s$'s 
and $\bigcup_{i=1}^n f_i(Z)$.
There is a morphism of $1$-motives
$$\delta:M \to A(D,T).$$
which is the identity on $\Alb^0(D)$. The 
induced morphism on lattices
$\mathcal{F}=\bigoplus_{s=1}^r\Z_{tr}(l)e_s
\to {\rm Div}^0_T(D)$
sends $e_s$ to the $0$-cycle $g_s$.
Also, the arrows 
$f_i:(C,Z,1) \to (D,T,1)$ in ${}^D(Crv_k)$ induces 
a morphism of $1$-motives
$$\gamma=\sum_{i=1}^n A(f_i):A(C,Z) \to A(D,T).$$
However, the triangle 
$$\xymatrix@C=1.7pc@R=1.7pc{A(C,Z) \ar[r]^-{\alpha} 
\ar@/_/[rd]^-{\gamma} & M \ar[d]^-{\delta} \\
& A(D,T)}$$
is not necessarily commutative. Let 
$\epsilon\df\gamma-\delta\circ \alpha$.
This is a morphism of $1$-motives 
such that the component $\epsilon_1:\Alb^0(C) \to \Alb^0(D)$ is zero.

In $\mathcal{M}_1^{\Q}$, we may decompose $\Alb(C,Z)=\mathcal{I}\oplus N$ 
where $N$ is a $1$-motive 
$[\mathcal{L}\hookrightarrow \Alb^0(C)]$
given by an injective morphism of group schemes. From our assumption on $M$, 
we have $\mathcal{I}\subset \ker(\alpha)$
and $\ker(\alpha|_N)$
is of the form $[\mathcal{L}\cap \mathcal{G} \hookrightarrow
\mathcal{G}]$ where $\mathcal{G}=\ker\{\alpha_1:\Alb^0(C) \to \Alb^0(D)\}$. It follows that 
\begin{equation}\label{eq:ker-gamma-N-alpha-N}
\ker(\gamma|_N)\subset \ker(\alpha|_N).
\end{equation}
Indeed, both $1$-motives in \eqref{eq:ker-gamma-N-alpha-N}
have the same semi-abelian part. 

Now, consider the sub-$1$-motive $\mathcal{T}\subset A(D,T)$ 
given by 
$$\gamma(\ker(\alpha|_N))=\epsilon(\ker(\alpha|_N)).$$
As $\epsilon$ is zero on the semi-abelian part, 
we see that $\mathcal{T}$ is a lattice. Also, using
\eqref{eq:ker-gamma-N-alpha-N},
we get 
$$\ker(\alpha|_N)=(\gamma|_{N})^{-1}(\mathcal{T}).$$
It follows that 
\begin{equation}\label{ker-alpha-formula}
\ker(\alpha)=\mathcal{I}+\gamma^{-1}(\mathcal{T}).
\end{equation}
Thus, we are left to show that $T_{\Q}(\mathcal{I})$ and
$T_{\Q}(\gamma^{-1}(\mathcal{T}))$
are sub-$\cR_1$-modules of $T_{\Q}(A(C,Z))$.

For $T_{\Q}(\mathcal{I})$, this follows from 
Lemma \ref{lemma-key-prop-5}.
For the second one, remark that 
$T_{\Q}(\gamma^{-1}(\mathcal{T}))$ is nothing but the inverse image of 
$T_{\Q}(\mathcal{T})\subset T_{\Q}(A(D,T))$ by the map
$$\sum_{i=1}^n a_if_{i*}: H_1(C,Z;\Q) \to 
H_1(D,T;\Q).$$
The latter being a morphism of $\cR_1$-modules, we may again apply 
Lemma \ref{lemma-key-prop-5}
to conclude.

\section{Nori versus Deligne $1$-motives}

Using \eqref{compare-M0s} and Theorem
\ref{thm:main-thm}, we define a 
functor from Deligne $1$-motives to Nori $1$-motives:
\begin{equation}\label{eqn:cM1-to-EHM1}
\nu_1:{}^t\cM_1\simeq \M_1''\to \M_1.
\end{equation}
The main result of the paper is:

\begin{thm}\label{Del=Nori}
The functor $\nu_1:{}^t\cM_1
\to \M_1$ is an equivalence of categories.
\end{thm}

The goal of this section is to show that the functor $\nu_1$ in 
\eqref{eqn:cM1-to-EHM1}
is fully faithful. In order to do this, 
we use that 
${}^t\cM_1$ embeds fully faithfully inside 
the category of mixed realizations.
Then, we reduce Theorem
\ref{Del=Nori} to showing that 
the essential image of ${}^t\cM_1$ is thick
in $\M$ (see the key Lemma~\ref{lem:iff-cond-equiv} below).
The proof of the latter property will be the subject of the next sections (completed in 
Section~\ref{last-sect-Deligne}).

\subsection{Mixed realisations}  Fix an embedding 
$\sigma:k\hookrightarrow \C$. 
We consider a variant, which we denote $\cM\cR^{\sigma}(k)$, 
of the category of mixed 
realisations (see \cite{DE} and \cf \cite{HU}) 
where, roughly speaking, we only retain the 
Betti component corresponding to $\sigma$, 
the de Rham component, and the 
$\ell$-adic components corresponding to the 
algebraic closure of $\sigma(k)$ in $\C$.
More specifically, an objet of $\cM\cR^{\sigma}(k)$ 
is a tuple 
$M\df (M_{\rm B}, M_{\rm dR}, M_{\ell},\dots)$ 
consisting of:
\begin{itemize}

\item a finitely generated abelian group
$M_{\rm B}$ together with an increasing filtration $W_{\d}$
on $M_{\rm B}\otimes \Q$, 
called the weight filtration,

\item a finitely generated $k$-vector space $M_{\rm dR}$
together with a decreasing filtration 
$F^{\d}$, called the Hodge filtration,

\item for every prime $\ell$, a finitely generated 
$\Z_{\ell}$-module $M_{\ell}$ together with a continuous action 
of the Galois group of $\bar{k}/k$, where $\bar{k}\subset \C$ 
is the algebraic closure of $k$ in $\C$, 

\item a comparison isomorphism 
$M_{\rm B}\otimes\C\simeq M_{\rm dR}\otimes_k\C$ such
that $(M_{\rm B},M_{\rm dR}\otimes_k\C,W_{\d},F^{\d})$
is a polarizable mixed Hodge structure,

\item for every prime $\ell$, 
a comparison isomorphism 
$M_{\rm B}\otimes\Z_{\ell}\simeq M_{\ell}$.

\end{itemize}
It is known that 
$\cM\cR^{\sigma}(k)$ is an abelian category (see \cite{DE} and \cf \cite{HU});
this is actually an easy consequence of the fact that 
mixed Hodge structures form an abelian category.

The following two simple remarks are useful. 

\begin{remark}
\label{rem:part-proj}
Projections yield functors from 
$\cM\cR^{\sigma}(k)$
to $\Zmod$ as well as ${\sf MHS} =\{\text{polarizable mixed Hodge structures}\}$
and $G_{k}-{\sf Rep}_{\ell}=
\{\ell-\text{adic Galois representations}\}$
where $G_k = \Gal(\bar{k}/k)$. The first two functors are faithful.
The third one is faithful up to $\ell'$-torsion.
\end{remark} 

\begin{remark}\label{rem:MR-base-change}
Given an extension $k'/k$ and 
a complex embedding $\sigma':k'\hookrightarrow k$ extending 
$\sigma$, one has a base-change functor 
$$-\otimes_kk':\cM\cR^{\sigma}(k)\to \cM\cR^{\sigma'}(k')$$
which is also faithful. If $M$ is a mixed realisation over $k$, then 
$M\otimes_kk'$ is simply given by
$$(M_{\rm B},M_{\rm dR}\otimes_kk', M_{\ell},\cdots)$$
where the action of $G_{k'}$ on $M_{\ell}$ is deduced from the 
action of $G_k$ by 
restricting along the canonical morphism
$G_{k'}\to G_k$.   
\end{remark}

The following is a variant of \cite[2.2 \& 2.3]{DE} for $1$-motives with torsion.

\begin{propose} \label{1mr}
Considering $T\df (T_{\Z}, T_{\rm dR}, T_{\ell}, \dots)$ where 
$T_{\Z}$ denotes the Betti realisation, $T_{\rm dR}$ the de Rham realisation, $T_{\ell}$ the $\ell$-adic realisation, etc.,  of $1$-motives with torsion, we obtain a functor 
\begin{equation}\label{1mrf}
T:{}^t\cM_1(k)\to \cM\cR^{\sigma}(k)
\end{equation} 
which is exact and fully faithful.
\end{propose}

\begin{proof}
We split the proof in three steps.

\smallskip

\noindent
\emph{Step 1:}
If $k=\C$, it is well-known \cite[Proposition 1.5]{BRS} that the composition 
$${}^t\cM_1(\C)\to \cM\cR^{\rm id}(\C)\to \mathsf{MHS}$$
is fully faithful.
From Remark
\ref{rem:part-proj},
we know that the second functor is faithful. 
This implies that the first functor is 
fully faithful.  

\smallskip

\noindent
\emph{Step 2:} 
If $k=\bar k$ is algebraically closed, 
the base change functor 
$$-\otimes_k\C: {}^t\cM_1(k)\to {}^t\cM_1(\C)$$
is fully faithful. (This easily follows from the case of lattices with torsion and semi-abelian varieties.) From Step 1, we deduce 
that the composition of
$${}^t\cM_1(k)\to {}^t\cM_1(\C) \to \cM\cR^{\rm id}(\C)$$
which is also the composition of
$${}^t\cM_1(k)\to \cM\cR^{\sigma}(k) \to \cM\cR^{\rm id}(\C)$$
is fully faithful. 
Now, by Remark
\ref{rem:MR-base-change}, the functor 
$$-\otimes_k\C: \cM\cR^{\sigma}(k)\to \cM\cR^{\rm id}(\C)$$ 
is faithful (in fact, it is also full, but we don't need to know this). This implies that 
${}^t\cM_1(k)\to \cM\cR^{\sigma}(k)$ is fully faithful.

\smallskip

\noindent
\emph{Step 3:}
We now consider the general case. Let $\bar{k}$ 
be the algebraic closure of $k$ in $\C$
and denote $G_k = \Gal(\bar{k}/k)$ the absolute Galois group.
Fix two $1$-motives $M$ and $M'$ over $k$
and denote $N$ and $N'$ their mixed realisations. 
Consider the following commutative diagram
$$\xymatrix{\Hom_{\,{^t\cM}_1}(M,M') \ar[r] \ar[d]& \Hom_{\,{^t\cM}_1}(M\otimes_k\bar{k},M'\otimes_k\bar{k})\ar[d]^-{\sim} \\
\Hom_{\cM\cR^{\sigma}}(N,N') \ar[r] \ar@{^(->}[d] & \Hom_{\cM\cR^{\sigma}}(N\otimes_k \bar{k},N'\otimes_k \bar{k})\ar@{^(->}[d] \\
\prod_{\ell}\Hom_{G_k}(N_{\ell},N'_{\ell})\ar[r] & 
\prod_{\ell}\Hom(N_{\ell},N'_{\ell})}$$
where the two lower vertical arrows are injective and 
the first vertical arrow on the right is invertible. 
Using a diagram chasing, it is enough 
to show that the commutative square
$$\xymatrix{
\Hom_{\,{^t\cM}_1}(M,M') \ar[r] \ar[d]& \Hom_{\,{^t\cM}_1}(M\otimes_k\bar{k},M'\otimes_k\bar{k})\ar@{^(->}[d]\\
\prod_{\ell}\Hom_{G_k}(N_{\ell},N'_{\ell})\ar[r] & 
\prod_{\ell}\Hom(N_{\ell},N'_{\ell})}$$
is cartesian.
The group $G_k$ acts on
$\Hom_{\,{^t\cM}_1}(M\otimes_k\bar{k},M'\otimes_k\bar{k})$
and 
$\Hom(N_{\ell},N'_{\ell})$
and the vertical arrow on the right is 
$G_k$-equivariant. 
Moreover, we have
$$\Hom_{G_k}(N_{\ell},N'_{\ell})=\Hom(N_{\ell},N'_{\ell})^{G_k}.$$
Thus, we are reduced to showing that the natural 
mapping 
\begin{equation}
\label{1mginv}
\Hom_{\,{}^t\cM_1}(M,M') \to \Hom_{\,{}^t\cM_1}(M\otimes_k\bar{k},M'\otimes_k\bar{k})^{G_k}
\end{equation}
is a bijection.
This follows from the
Hoschschild-Serre spectral sequence.
Indeed, from \cite[Theorem 2.1.2]{BK}, 
one has a fully faithful embedding 
$$\Tot:{\bf D}^b({}^t\cM_1(k))\to \DM_{\eff}^{\et}(k)$$ 
where $\DM_{\eff}^{\et}(k)$ is the full subcategory 
of ${\bf D}(Shv^{\et}_{tr}(k))$ given by $\Aff^1$-local objects
(\cf \S\ref{subsect:VoMo} \& Proposition
\ref{prop:orgo-embed} below for a more detailed discussion, 
but only with rational coefficients). 
Also, there is a similar functor for $\bar{k}$.
Therefore
$$\Hom_{\,{}^t\cM_1}(M,M')=\Hom_{\DM_{\eff}^{\et}}(\Tot(M),\Tot(M'))$$ 
and similarly after applying $-\otimes_k\bar{k}$.
Now in ${\bf D}(Shv^{\et}_{tr}(k))$ as well as in $\DM_{\eff}^{\et}(k)$ 
we have that 
$$\RHom(C,C')=R\Gamma(G_k,\RHom(C\otimes_k\bar{k},C'\otimes_k\bar{k}))$$
for objects $C, C'\in {\bf D}(Shv^{\et}_{tr}(k))$. 
For $C=\Tot (M)$ and $C'=\Tot (M')$, this gives
\eqref{1mginv}.
\end{proof}

\subsection{Mixed realisation of effective Nori motives}
Considering 
$$(X,Y, i)\leadsto R_i(X,Y)=(H_i^{\rm B}(X, Y), H_i^{\rm dR}(X, Y), H_{i}^{\ell}(X, Y),\dots)$$
given by singular homology, de Rham homology and $\ell$-adic 
homology, we get a representation 
$R:{}^D(Sch_k)\to \cM\cR^{\sigma}(k)$ 
which factors the representation 
\eqref{can-rep-sch-k}. By universality 
(\ie \cite[Theorem 41]{LV}), we obtain 
an exact faithful functor
\begin{equation}\label{RN}
\widetilde{R} : \M(k) \to \cM\cR^{\sigma}(k).
\end{equation}

\begin{lemma} \label{repalb} 
For $(C,Z,1)\in {}^D(Crv_k)$, there is a canonical isomorphism 
$T\circ A(C,Z) \simeq R(C,Z,1)$. In other words, 
the following square is commutative:
$$\xymatrix{{}^D(Crv_k) \ar[d] \ar[r]^-{A} \ar[d]
& {}^t\mathcal{M}_1\ar[d]^-{T}\\
{}^D(Sch_k) \ar[r]^{R} & \mathcal{MR}^{\sigma}.\!}$$
\end{lemma}

\begin{proof} Since we deal with curves this is essentially due to Deligne \cite[\S 10.3]{HdgIII}. (Deligne deals with the cohomology of 
curves: one needs to dualize to get the statement we need.)
\end{proof} 
 
\begin{propose}\label{M-dprime-M}
The exact functor 
$\M''_1\to \M$ is fully faithful.
\end{propose}

\begin{proof}
Indeed, the commutative square in Lemma 
\ref{repalb} gives a commutative square of 
exact faithful functors 
$$\xymatrix{\M_1'' \ar[r]^-{\sim} \ar[d] 
& {}^t\mathcal{M}_1 \ar[d] \\
\M\ar[r] & \mathcal{MR}^{\sigma}}$$
where the upper horizontal arrow is an 
equivalence by Theorem 
\ref{thm:main-thm}
and the right vertical arrow is fully faithful by 
Proposition \ref{1mr}. This proves the claim.
\end{proof}

\begin{cor}\label{cor:iff-cond-equiv}
The functor 
\begin{equation}\label{eq:iff-cond-equiv}
\nu:{}^t\mathcal{M}_1\simeq \M_1''\to \M
\end{equation}
is fully faithful and the triangle 
$$\xymatrix{{}^t\mathcal{M}_1\ar[r]^-{\nu} \ar@/_/[dr]_-{T} & \M \ar[d]^-{\widetilde{R}} \\
& \mathcal{MR}^{\sigma}}$$
is commutative up to a canonical isomorphism.
Moreover, $\M_1$ is the thick abelian subcategory 
generated by the image of $\nu$. 
\end{cor}

\begin{proof}
That $\nu$ is fully faithful follows 
from Proposition
\ref{M-dprime-M}. Also, the equality $T=\widetilde{R}\circ \nu$ is 
clear from the construction.
It remains to show that $\M_1$ is the thick 
abelian subcategory of $\M$ generated by the image of 
$\nu$. By Proposition 
\ref{generate}, it is enough to show that 
the essential image of $\nu$ contains the motives
$\widetilde{H}_i(C,Z,\Z)$ for $i\leq 1$ and 
$C$ a $k$-scheme with ${\rm dim}(C)\leq 1$.

By Proposition
\ref{N0}, this is clear for $i=0$. Thus, we may assume that
$i=1$.
If $(C,Z,1)\in {}^D(Crv_k)$, then the property
we need follows from the construction. We reduce the general
case to the previous one as follows.
We may assume that ${\rm dim}(Z)=0$. If $Z'\subset C$ is a zero dimensional sub-scheme 
containing $Z$, then $\widetilde{H}_1(C,Z,\Z)\to 
\widetilde{H}_1(C,Z',\Z)$ is injective.
Therefore we may enlarge $Z$ and 
assume that $C\smallsetminus Z$ is smooth. 
If $C''$ is the normalization of $C$ and $Z''$ is the inverse image 
of $Z$ by $C''\to C$, we have 
$\widetilde{H}_1(C'',Z'',\Z)\simeq \widetilde{H}_1(C,Z,\Z)$. 
Thus, we may assume that $C$ is smooth. Finally if $C$ is complete, 
and $c\in C\smallsetminus Z$ a closed point, the morphism
$\widetilde{H}_1(C\smallsetminus \{c\},Z;\Z)\to 
\widetilde{H}_1(C,Z;\Z)$ is surjective. Therefore, we may assume that 
$C$ is affine. This finishes the proof.
\end{proof}

\smallskip

\begin{lemma}\label{lem:iff-cond-equiv}
The following 
conditions are equivalent:
\begin{enumerate}

\item[(a)] $\nu_1$ is an equivalence of categories 
(\cf\eqref{eqn:cM1-to-EHM1}),

\item[(b)] the essential image of $\nu$ is a thick abelian subcategory of $\M$. 

\end{enumerate}
\end{lemma}

\begin{proof}
This follows from Corollary 
\ref{cor:iff-cond-equiv}.
\end{proof}

\section{Some reductions}

In this section we start the verification of 
\ref{lem:iff-cond-equiv}(b). (By Lemma
\ref{lem:iff-cond-equiv}, this is what we still need to prove in order to complete 
the proof of Theorem \ref{Del=Nori}.)
We will see here that the essential image of the functor $\nu$ in 
\eqref{eq:iff-cond-equiv} is stable under sub-quotients
in $\M$. Stability by extensions will be the subject of 
Section 
\ref{last-sect-Deligne}.

\begin{propose}\label{prop:cM1-st-squo}
The essential image of the fully faithful exact functor 
$T:{}^t\cM_1(k)\to \mathcal{MR}^{\sigma}(k)$ 
is stable under sub-quotients. 
\end{propose}

\begin{proof}
It 
suffices to prove stability by sub-objects. 
Fix a $1$-motive $M$ and a sub-object $N'\subset N$ of 
its mixed realization 
$N=T(M)$.
We need to construct a sub-$1$-motive $M'\subset M$ such that
$T(M')=N'$.
As for Proposition
\ref{1mr}, we split the proof in three steps.

\smallskip

\noindent
\emph{Step 1:}
If $k=\C$, we know that the composition 
$${}^t\cM_1(\C)\to \cM\cR^{\rm id}(\C)\to \mathsf{MHS}$$
induces an equivalence between ${}^t\cM_1(\C)$ and 
the subcategory of ${\sf MHS}$ 
consisting of mixed hodge structures of type
$$\{(0,0), (0,-1), (-1,0), (-1,-1)\}.$$
The latter is a thick abelian subcategory of ${\sf MHS}$. 
Applying this to the mixed Hodge structure determined by $N'$, we find
a sub-$1$-motive $M'\subset M$ such that the sub-objects
$T(M')\subset T(M)$ and $N'\subset T(M)$ determine the same sub-mixed Hodge structure. This implies that 
$T(M')=N'$.

\smallskip

\noindent
\emph{Step 2:}
If $k=\bar k$ is algebraically closed, 
we know that the base change functor 
$$-\otimes_k\C: {}^t\cM_1(k)\to {}^t\cM_1(\C)$$
is fully faithful. Moreover, its essential image is stable 
under sub-objects. This follows from the fact that, for 
a lattice $\mathcal{L}$ (resp. a semi-abelian variety 
$\mathcal{G}$) defined over $k$, every sub-lattice 
of $\mathcal{L}\otimes_k\C$ 
(resp. sub-semi-abelian variety of 
$\mathcal{G}\otimes_k\C$) is defined over $k$.
Therefore, to construct a sub-$1$-motive of $M$ that realizes to $N'$, it is enough to construct a sub-$1$-motive of 
$M\otimes_k\C$ that realizes to $N'\otimes_k\C$. 
We use the previous step to conclude.

\smallskip

\noindent
\emph{Step 3:}
We now consider the general case. Let $\bar{k}$ 
be the algebraic closure of $k$ in $\C$.
By the previous step, we may find a sub-$1$-motive 
$M_{\bar{k}}'\subset M\otimes_k\bar{k}$ such that
$T(M_{\bar{k}}')=N'\otimes_k\bar{k}$.

The sub-$1$-motive $M'_{\bar{k}}$ can be defined over 
a finite Galois extension $l\subset \bar{k}$, \ie there exists a
sub-$1$-motive $M'_l\subset M\otimes_kl$ such that 
$$T(M'_l\otimes_l\bar{k})=N'\otimes_k\bar{k}.$$
This implies  
that, at least, $T_{\rm dR}(M'_l)=N_{\rm dR}'\otimes_kl$.

Let $\tilde{M'}\in {}^t\cM_1(k)$ be the Weil restriction of 
$M'_l\in {}^t\cM_1(l)$. This is a sub-$1$-motive of 
$\tilde{M}=M\otimes \Z_{tr}(l)$. It is characterized by the property 
that 
$$\tilde{M}'\otimes_kl=\bigoplus_{\tau\in \hom_k(l,l)}M'_l\otimes_{l,\tau}l$$
as a sub-$1$-motive of 
$$\tilde{M}\otimes_kl=\bigoplus_{\tau\in \hom_k(l,l)}M\otimes_kl.$$
There is a canonical morphism of $1$-motives
$M\to \tilde{M}$ and we define $M'$ to be the intersection 
of $M$ and $\widetilde{M}'$ inside $\widetilde{M}$.
Then, by construction, we have $M'\otimes_kl\subset M'_{l}$.
Therefore, we also have $T(M')\subset N'$.

Now, by construction 
$$T_{\rm dR}(\widetilde{M})=T_{\rm dR}(M\otimes_kl)=T_{\rm dR}(M)\otimes_kl,$$ 
viewed as a 
$k$-vector space.
Moreover, 
$$T_{\rm dR}(\tilde{M}')={\rm T}_{\rm dR}(M_l')=N'_{\rm dR}\otimes_kl,$$ viewed
as a sub-$k$-vector space of 
$T_{\rm dR}(M\otimes_kl)$. As $T_{\rm dR}$ is an exact functor, 
it follows from the construction of $M'$ that 
$$T_{\rm dR}(M')=T_{\rm dR}(M)\cap (N'_{\rm dR}\otimes_kl) 
\qquad \text{inside} \quad T_{\rm dR}(M)\otimes_kl.$$
This shows that $T_{\rm dR}(M')=N'_{\rm dR}$.
Therefore,
$T(M')\subset N'$ has finite index.

Replacing $M$ by $M/M'$ and $N'$ by $N'/T(M')$, we may assume that 
$N'$ has zero de Rham component, \ie that
$N'$ is a torsion object of $\mathcal{MR}^{\sigma}$. 
In particular, $N'$ lies in the essential image of $\mathcal{M}_0\to \mathcal{MR}^{\sigma}$: let $\mathcal{L}$ be a torsion lattice (\ie a finite \'etale commutative group scheme over $k$)
such that $N'=T([\mathcal{L}\to 0])$. 
We may use 
Proposition
\ref{1mr}
to find a monomorphism $[\mathcal{L}\to 0]\hookrightarrow M$ 
that realizes to $N'$ as a sub-object of $N$. This finishes the proof.
\end{proof}

\begin{cor}\label{cor:imv-st-squot}
The essential image of the functor 
$\nu:{}^t\mathcal{M}_1\to \M$ is an abelian subcategory 
that is stable under sub-quotients.
\end{cor}

\begin{proof}
Indeed, let $M$ be a $1$-motive and let
$L'\subset L$ be a sub-motive of the effective Nori motive $L=\nu(M)$.
Set $N=T(M)=\tilde{R}(\nu(M))$ and $N'=\tilde{R}(L')$.
Then $N'$ is a sub-object of
$N$ and, by Proposition
\ref{prop:cM1-st-squo}, there exists a sub-$1$-motive
$M'\subset M$ such that $T(M')=N'$.  
Then, necessarily $\nu(M')=L'$ as sub-objects of $\nu(M)$. 
Indeed, this can be checked after applying 
the forgetful functor $\M\to \Zmod$.
\end{proof}

\begin{lemma}\label{red:to-Del-ext}
To check \ref{lem:iff-cond-equiv}(b), it is enough 
to check that
the essential image of $\mathcal{M}_1^{\Q}$ by $\nu\otimes \Q$ is
stable under extensions in $\M^{\Q}$. 
\end{lemma}

\begin{proof}
In view of Corollary
\ref{cor:imv-st-squot}, it remains to check that 
the essential image of ${}^t\mathcal{M}_1$ by $\nu$ is stable under extensions in 
$\M$. We need to prove this property assuming its 
rational analogue. 
So, consider an exact sequence in $\M$:
\begin{equation}\label{ext-nu-M-N-M-1}
0\to \nu(M')\overset{r}{\to} N \overset{s}{\to} \nu(M'') \to 0
\end{equation}
where $M'$ and $M''$ are $1$-motives. 
To show that $N$ is in the essential image of $\nu$, we 
consider several special cases.

\smallskip

\noindent
\emph{Case 1:} $M'$ is torsion-free and 
$M''$ is torsion.
Then $\nu(M')\to N$ is an isomorphism in 
$\M^{\Q}$. Therefore, there exists a morphism 
$t:N\to \nu(M')$ such that the composition 
$t\circ r$ 
is a multiplication by an non zero integer. 
Consider the morphism 
$(t,s):N\to \nu(M')\oplus \nu(M'')$.
As $M'$ is torsion-free, this morphism is injective. 
Hence, we have realized $N$ as a sub-object of the image 
of a $1$-motive by $\nu$. By Corollary
\ref{cor:imv-st-squot} we are done.

\smallskip

\noindent
\emph{Case 2:} both $M'$ and $M''$ are torsion-free.
In this case $N$ is also torsion-free (\ie its Betti 
realization is a free $\Z$-module).
Using the assumption in the statement, there exists 
a $1$-motive $M$ and an isomorphism 
$N\simeq \nu(M)$ in $\M^{\Q}$. Let 
$N\to \nu(Q)$ be a morphism in $\M^{\Q}$ inducing this isomorphism. 
As $N$ is torsion-free, this morphism is injective. Again, 
we are done by Corollary
\ref{cor:imv-st-squot}.

\smallskip

\noindent
\emph{Case 3:} $M'$ is torsion-free and $M''$ is general.
For every finite extension $l/k$, the exact sequence 
\eqref{ext-nu-M-N-M-1}
induces an exact sequence
$$0\to \nu(M'\otimes \Z_{tr}(l))\to N\otimes \nu(\Z_{tr}(l)) \to
\nu(M''\otimes \Z_{tr}(l)) \to 0$$
and, there is an injective morphism 
$N\hookrightarrow N\otimes 
\nu(\Z_{tr}(l))$.
Using Corollary
\ref{cor:imv-st-squot}, it is enough to show that 
$N\otimes 
\nu(\Z_{tr}(l))$ is in the essential image of $\nu$. 
Therefore, we may replace $M'$ and $M''$ by 
$M'\otimes\Z_{tr}(l)$ and $M''\otimes\Z_{tr}(l)$, for $l/k$ large enough, and assume that the torsion part of the lattice 
$\mathcal{L}''$ of $M''$ is a direct summand.
This implies that $M''=M''_t\oplus M''_f$ where 
$M''_t$ is a torsion $0$-motive and $M''_f$ is a torsion-free
$1$-motive. It follows that $N$ embeds in a direct sum 
$N_t\oplus N_f$ where $N_t$ and $N_f$ are extensions 
of $\nu(M')$ by $\nu(M''_t)$ and $\nu(M''_f)$. 
By cases $1$ and $2$, 
$N_t$ and $N_f$ are in the essential image of $\nu$. 
Again, 
we are done by Corollary
\ref{cor:imv-st-squot}.

\smallskip

\noindent
\emph{Case 4:} $M'$ and $M''$ are general. 
Let $M'_t\subset M'$ be the torsion part of $M'$. 
By the previous case, $N/\nu(M'_t)$ is in the essential image of $\nu$. 
In other words, we may change the exact sequence 
\eqref{ext-nu-M-N-M-1}
and assume that $M'$ is torsion.

It follows that $N\to \nu(M'')$ is an isomorphism 
in $\M^{\Q}$.
Therefore, there exists a torsion-free $1$-motive 
$M'''$ and a morphism 
$t:\nu(M''') \to N$ such that 
the composition $s\circ t$ is given by an isogeny 
from $M'''$ to $M''$.
Using case 1, it is enough to show that 
$N/\nu(M''')$ is in the essential image of $\nu$. 
The latter being an extension of $\nu(M')$ by 
$\nu(M''/M''')$, we are reduced to the case where 
$M''$ is also torsion. Now, by a result of Nori
\cite{NN}, 
all torsion objets in $\M$ are $0$-motives.
This finishes the proof of the lemma.
\end{proof}

\section{On Deligne's conjecture on extensions of $1$-motives}

\label{last-sect-Deligne}

In this section, we prove a ``piece'' of Deligne's conjecture 
\cite[2.4]{DE} on 
extensions of $1$-motives that is needed to complete the proof
of Theorem \ref{Del=Nori}.
More precisely, we prove Deligne's conjecture under 
an effectivity condition coming from Nori's formalism of mixed motives.
(This condition plays a crucial role in our argument 
and seems difficult to remove 
with the actual motivic technology.)

\subsection{Voevodsky motives}\label{subsect:VoMo}
Let $\DM_{\eff}(k;R)$ be Voevodsky's category of motives
with coefficients in $R$. 
There is a fully faithful embedding:
$$\DM_{\eff}(k;R) \hookrightarrow 
\mathbf{D}(Shv^{\rm Nis}_{tr}(k;R)).$$
Its image consists of those complexes of 
Nisnevich sheaves with transfers which are 
$\Aff^1$-local, \ie such that their Nisnevich hypercohomology
presheaves 
are $\Aff^1$-invariant. 

As usual, we denote $\DM_{\eff}^{\gm}(k;R)$ the 
subcategory of $\DM_{\eff}(k;R)$ of geometric motives.
It is the thick triangulated subcategory generated by 
motives of 
smooth $k$-schemes.

Given a pair $(X,Y)$ where 
$X$ is a $k$-scheme and $Y\subset X$ a closed subset, 
one has an object ${\rm M}(X,Y;R)\in \DM_{\eff}(k;R)$, called 
the motive of the pair. (As a presheaf with transfers, this is 
simply given by $R_{tr}(X)/R_{tr}(Y)$.) 
Since our base field has characteristic zero, 
we know, thanks to Hironaka's resolution of singularities,
that ${\rm M}(X,Y;R)$ is a geometric motive. 

Another source of examples is given by the following 
result (claimed by Voevodsky in 
\cite[\S 3.4]{V} and proved in \cite{Or}).

\begin{propose}\label{prop:orgo-embed}
There exists a fully faithful embedding 
\begin{equation}\label{orgo-embed}
{\rm Tot}:{\bf D}^b(\mathcal{M}_1^{\Q}(k)) \to 
\DM^{\gm}_{\eff}(k;\Q).
\end{equation}
It induces an equivalence of categories with  
the thick triangulated subcategory $\DM^{\gm}_{\leq 1}(k;\Q)$
generated by motives of curves.  
\end{propose}

\begin{proof}
The functor is easily defined: it sends a $1$-motive 
$[\mathcal{L}\to \mathcal{G}]$ 
to the complex of Nisnevich sheaves with transfers
$[\mathcal{L}\otimes \Q\to \mathcal{G}\otimes \Q]$
placed in homological degrees $0$ and $-1$.
For details concerning the proof, we refer the reader to 
\cite{Or}.
\end{proof}

\subsection{Voevodsky versus Nori}
Nori \cite{NN} constructed a triangulated functor
\begin{equation}\label{Noris-Gamma}
\Gamma:\DM_{\eff}^{\gm}(k;R)\to {\bf D}^b(\M^R(k)).
\end{equation}
This functor transforms ${\rm M}(X,Y;R)$ 
into a complex of effective Nori motives $\Gamma(X,Y;R)$ such that,
canonically, 
\begin{equation}\label{char-Nori-to-Voe}
H_i(\Gamma(X,Y;R))\simeq \widetilde{H}_i(X,Y;R).
\end{equation}
(In the left hand side of the above formula, $H_i(-)$ 
is the homology functor with respect to the canonical 
$t$-structure on ${\bf D}^b(\M^R(k))$.)

Nori's functor 
\eqref{Noris-Gamma}
can be used to recover all the 
realization functors on 
$\DM_{\eff}^{\gm}(k,\Q)$ constructed by Huber \cite{HUR}. 
For instance, one gets a mixed realization functor on 
$\DM_{\eff}^{\gm}(k;R)$ 
by taking the composition
$$\DM_{\eff}^{\gm}(k;\Z)\overset{\Gamma}{\to} {\bf D}^b(\M(k)) \overset{\widetilde{R}}{\to} 
{\bf D}^b(\mathcal{MR}^{\sigma}(k)).$$
(The verification that this composition is 
isomorphic to Huber's functor 
is tedious but routine; it will not be carried out in this paper.
Happily, this is not needed for any of our main results.)

\begin{propose}\label{new-Volog}
The following square commutes up to a natural isomorphism
$$\xymatrix{\mathcal{M}_1^{\Q}(k) \ar[r]^-{\nu} \ar[d]^-{\rm Tot} & \M^{\Q}(k)\ar[d] 
\\
\DM_{\eff}^{\gm}(k;\Q) \ar[r]^-{\Gamma} & {\bf D}^b(\M^{\Q}(k)).\!}$$
\end{propose}

\begin{proof}
The image of the composition of 
$$\mathcal{M}_1^{\Q}(k) \to \DM_{\eff}^{\gm}(k;\Q) \to 
{\bf D}^b(\M^{\Q}(k))$$
lies in the heart of the canonical $t$-structure on 
${\bf D}^b(\M^{\Q}(k))$. To check this, it is 
enough to prove the same claim for the composition of
$$\mathcal{M}_1^{\Q}(k) \to \DM_{\eff}^{\gm}(k;\Q) \overset{R^{\rm B}}{\to} 
{\bf D}^b(\Q)$$
where $R^{\rm B}$ is the Betti realization.
Using the weight filtration on $1$-motives, it is enough to 
consider separately the 
case of a lattice, of a torus, and of an abelian variety. 
Also, 
we may assume that $k=\C$. Then, 
the first two cases are obvious.  
For the third case, we use that for an 
abelian variety $A$, ${\rm Tot}([0\to A])=A\otimes \Q[-1]$ 
is a direct factor of 
a motive ${\rm M}(C\smallsetminus\{c_1\}, c_2;\Q)[-1]$
where $C$ is a complete smooth curve, and $c_1$ and $c_2$ are
two distinct
rational points. The Betti realization of such a motive
is the complex $H_1^{\rm B}(C)[0]$.

Due to the previous discussion, it is enough to show that 
the following diagram 
$$\xymatrix{\mathcal{M}_1^{\Q}(k) \ar[r]^-{\nu} \ar[d] & \M^{\Q}(k)
\\
\DM_{\eff}^{\gm}(k;\Q) \ar[r]^-{\Gamma} & {\bf D}^b(\M^{\Q}(k)) \ar[u]_-{H_0}\!}$$
commutes.
Now we have to deal with two 
functors from 
$\mathcal{M}_1^{\Q}$ to $\M^{\Q}$, which are $\Q$-linear and exact.
Thus, by Theorem
\ref{thm:main-thm} and universality
(\ie \cite[Theorem 41]{LV}), it will be 
enough to check that 
$$\xymatrix{{}^D(Crv_k) \ar[r]^-{\widetilde{H}} \ar[d]^-{{\rm Tot}\,\circ\, A} & \M^{\Q}
\\
\DM_{\eff}^{\gm}(k;\Q) \ar[r]^-{\Gamma} & {\bf D}^b(\M^{\Q}) \ar[u]_-{H_0}\!}$$
commutes. 
Now, for $(C,Z,1)\in {}^D(Crv_k)$, 
it follows from \cite[Theorem 3.4.2]{V} that 
there exists an exact triangle in $\DM^{\gm}_{\eff}(k;\Q)$:
$${\rm Tot}(A(C,Z)) \to {\rm M}(C,Z;\Q)[-1] \to \mathcal{L}[-1] \to$$ 
where $\mathcal{L}=\coker\{\Q_{tr}(Z) \to \Q_{tr}(\pi_1(C))\}$.
Moreover, this triangle splits (non canonically). 
As $H_0(\Gamma(\mathcal{L}[-1]))=0$, we get 
an isomorphism 
$$H_0(\Gamma({\rm Tot}(A(C,Z))))\simeq H_0(\Gamma({\rm M}(C,Z;\Q))).$$
Therefore, in the last square above, we may replace 
${\rm Tot}\circ A$ by ${\rm M}$. 
The commutativity is then a direct consequence of 
\eqref{char-Nori-to-Voe}. 
\end{proof}

Although not needed for our main objective, 
we note the following concrete consequence of 
Proposition \ref{new-Volog}.

\begin{cor}\label{cor:new-volog}
The following square commutes up to a natural isomorphism
$$\xymatrix{\mathcal{M}_1^{\Q}(k) \ar[r]^-{T} \ar[d]^-{\rm Tot} & \mathcal{MR}^{\sigma}(k;\Q)\ar[d] 
\\
\DM_{\eff}^{\gm}(k;\Q) \ar[r]^-{\widetilde{R}\circ \Gamma} & {\bf D}^b(\mathcal{MR}^{\sigma}(k;\Q))}$$
(where we set 
$\mathcal{MR}^{\sigma}(k;\Q)=\mathcal{MR}^{\sigma}(k)\otimes \Q$).
\end{cor}

\begin{proof}
This is a direct consequence of Proposition 
\ref{new-Volog} and the 
equatity $T=\widetilde{R}\circ \nu$ 
(see Corollary 
\ref{cor:iff-cond-equiv}). 
\end{proof}

\begin{remark}\label{rem:new-volog}
As a consequence we obtain that
Deligne's Hodge realization of $1$-motives
is isomorphic to the composition of Huber's Hodge
realization with 
the embedding 
$\mathcal{M}_1^{\Q}(k)\hookrightarrow \DM_{\eff}^{\gm}(k;\Q)$.
This was first proved by Vologodsky \cite{VO}. Our proof 
is arguably more conceptual. 
\end{remark}

\subsection{\/}
Let ${\sf MHS}^{\Q}_{\eff}\subset {\sf MHS}^{\Q}$ 
denote the full subcategory 
of (homologically) effective mixed Hodge structures. Also, let 
${\sf MHS}^{\Q}_{\leq 1}$ be the thick abelian subcategory of 
${\sf MHS}_{\eff}^{\Q}$ consisting of mixed Hodge structures
of type 
$$\{(0,0),(-1,0),(0,-1),(-1,-1)\}.$$
We have a fully faithful embedding
${\bf D}^b({\sf MHS}_{\leq 1}^{\Q})\to 
{\bf D}^b({\sf MHS}^{\Q})$ and a commutative square
\begin{equation}\label{square-DM-MHS}
\xymatrix{\DM^{\gm}_{\leq 1}(k;\Q) \ar[d] \ar[rr]^{R^{\rm Hdg}|_{\leq 1}} & & {\bf D}^b({\sf MHS}_{\leq 1}^{\Q}) \ar[d] \\
\DM^{\gm}_{\eff}(k;\Q) \ar[rr]^-{R^{\rm Hdg}} && {\bf D}^b({\sf MHS}_{\eff}^{\Q}).\!}
\end{equation}
(For the sake of precision, we note that 
$R^{\rm Hdg}$ is taken to be the composition of 
$$\DM^{\gm}_{\eff}(k;\Q) \overset{\Gamma}{\to}
{\bf D}^b(\M^{\Q}(k)) \overset{R^{\rm Hdg}}{\longrightarrow} {\bf D}^b({\sf MHS}^{\Q}_{\eff})$$
where the second functor is derived from 
$R^{\rm Hdg}:\M^{\Q}\to {\sf MHS}^{\Q}_{\eff}$ 
given by the universal property.)

Both vertical inclusions in 
\eqref{square-DM-MHS} 
admit left adjoints (see \cite[Theorem 2.4.1]{ABV} or \cite[Corollary 6.2.2]{BK} for the first one and 
\cite[Proposition 17.1.1]{BK} for the second one); 
they are denoted respectively by 
$\LAlb$ and $(-)_{\leq 1}$. 
From the commutativity of 
\eqref{square-DM-MHS}, we get a natural transformation 
\begin{equation}\label{eq:R-Hodge-L-Alb}
(R^{\rm Hdg} (M))_{\leq 1} \to (R^{\rm Hdg}|_{\leq 1})(\LAlb (M))=
R^{\rm Hdg}(\LAlb (M)).
\end{equation}

\begin{propose}\label{First-Del-Conj}
The natural transformation 
\eqref{eq:R-Hodge-L-Alb} is invertible. 
\end{propose}

\begin{proof}
This is essentially \cite[Theorem 17.3.1]{BK}. For completeness, 
we give a sketch of the argument (with a slight modification).
By the proof of \cite[Corollary 2.4.6]{ABV}, 
it is enough to show that 
\eqref{eq:R-Hodge-L-Alb} is invertible after evaluating on
motives  
${\rm M}(X)$, where $X$ is a smooth 
$k$-scheme which is $\NS^1$-local (in the sense of 
\cite[Definition 2.4.2]{ABV}).
Using \cite[Proposition 2.4.4]{ABV}, we have $\LAlb(X)=\Alb(X)$, where the semi-abelian variety
$\Alb(X)$ is considered as a Nisnevich sheaf with transfers. 

Similarly, writing $H^{\rm Hdg}_i(X)$ for the mixed Hodge 
structure on the $i$-th homology of $X$, we have: 
$$(H^{\rm Hdg}_i(X))_{\leq 1}=
\left\{\begin{array}{cc}
H^{\rm Hdg}_i(X) & \text{if } i\in \{0,1\},\\
0 & \text{otherwise}.
\end{array}\right.$$
In other words, we have  
$$(R^{\rm Hdg}(X))_{\leq 1}=\tau_{\leq 1}R^{\rm Hdg}(X)$$
where $\tau_{\leq 1}$ is the good truncation with 
respect to the canonical $t$-structure on 
${\bf D}^b({\sf MHS}^{\Q}_{\eff})$.

Thus, to finish the proof, we are left to show that 
the Betti realization of 
${\rm M}(X) \to \Alb(X)$ 
is isomorphic to
$R^{\rm B}(X) \to \tau_{\leq 1}R^{\rm B}(X)$. 
But, the complex $R^{\rm B}(\Alb(X))$ has homology in degree 
$0$ and $1$ (see the first part of the proof of 
Proposition \ref{new-Volog}). Moreover, we have canonically: 
$$H_0(R^{\rm B}(\Alb(X)))=H_0^{\rm B}(X) \quad\text{and}\quad
H_1(R^{\rm B}(\Alb(X)))=H_1^{\rm B}(X).$$ 
This finishes the proof.
\end{proof}

\begin{thm}\label{thm:eff-Del-conj}
Let $M\in \M^{\Q}$ be an effective 
Nori motive whose Hodge realization is
in ${\sf MHS}^{\Q}_{\leq 1}$. Then, $M$ is in the essential image of
the functor $\nu:\mathcal{M}^{\Q}_1\to \M^{\Q}$.
\end{thm}

\begin{proof}
We can realize $M$ as a sub-quotient 
of a Nori motive of the
form $\tilde{H}_i(X,Y;\Q)$ with $X$ a $k$-scheme and 
$Y\subset X$ a closed subset. 
Consider the motive 
${\rm M}(X,Y;\Q)\in \DM^{\gm}_{\eff}(k;\Q)$ and set 
$$A=H_i(\LAlb({\rm M}(X,Y;\Q))).$$
(In the above formula, $H_i$ is with respect to the motivic 
$t$-structure on $\DM^{\gm}_{\leq 1}(k;\Q)$ deduced from the 
canonical $t$-structure on 
${\bf D}^b(\mathcal{M}_1^{\Q})$ via 
the equivalence in Proposition
\ref{prop:orgo-embed}.)

Set $N=\Gamma(A)$. By construction, we have 
a map 
\begin{equation}\label{HiXYQN}
\widetilde{H}_i(X,Y;\Q) \to N.
\end{equation}
It is obtained by applying $H_i$ to the 
obvious morphism
$$\Gamma({\rm M}(X,Y;\Q))\to \Gamma(\LAlb(M(X,Y;\Q))).$$
In particular, Proposition 
\ref{First-Del-Conj} implies that 
\eqref{HiXYQN} induces on the associated mixed 
Hodge structures, the obvious projection:
\begin{equation}\label{Hodge-R-Hi-to-N}
H_i^{\rm Hdg}(X,Y;\Q)\to (H_i^{\rm Hdg}(X,Y;\Q))_{\leq 1}.
\end{equation}
(As before, we denote 
$H_i^{\rm Hdg}(X,Y;\Q)$ the mixed Hodge structure on the 
rational $i$-th homology of the pair $(X,Y)$.)

Now, write $M=M'/M''$ where $M'$ and $M''$ are sub-motives of 
$\widetilde{H}_i(X,Y;\Q)$. Let $N'$ and $N''$ be the images of 
$M'$ and $M''$ in $N$. It follows, 
by looking at the 
associated mixed Hodge structures, that 
$M\simeq N'/N''$. This proves that 
$M$ is a sub-quotient of $\Gamma(A)$. 
Let $A'\in \mathcal{M}^{\Q}_1$ be a $1$-motive 
such that $A={\rm Tot}(A')$. By 
Proposition \ref{new-Volog}, we have 
$$\Gamma(A)=\Gamma({\rm Tot}(A'))=\nu(A').$$
Thus, we have realized $M$ as a sub-quotient 
of an object in the image of $\nu$. Now use Corollary~\ref{cor:imv-st-squot}. (In fact,
by the commutativity of the triangle in Corollary~\ref{cor:iff-cond-equiv} and Proposition 
\ref{prop:cM1-st-squo} we are done.)
\end{proof}

\begin{remark}\label{rem:expl-Del-conj-eff}
As ${\sf MHS}_{\leq 1}^{\Q}$ is a thick abelian 
subcategory of ${\sf MHS}^{\Q}_{\eff}$,
Theorem \ref{thm:eff-Del-conj}
gives a positive answer to Deligne's conjecture
\cite[2.4]{DE} in the case where the ``geometric'' extension of 
$1$-motives is \emph{effectively geometric}, 
\ie lying in the image of the 
functor $\widetilde{R}:\M\to \mathcal{MR}^{\sigma}$ from the 
category of \emph{effective} Nori motives. 

Let $\mathsf{NHM}$ be the category of (non necessarily effective) 
Nori motives. It is natural to expect that 
Theorem 
\ref{thm:eff-Del-conj} holds more generally for 
$M\in \mathsf{NHM}$. This would give a positive 
answer to Deligne's conjecture \cite[2.4]{DE}
in full generality. On the other hand, 
it is reasonable to expect that 
$\M\subset {\sf NHM}$ is a thick abelian category. 
However, such a statement is completely out of reach and goes 
far beyond Deligne's conjecture.
\end{remark}

\begin{cor}
The essential image of $\nu:\mathcal{M}^{\Q}_1\to 
\M^{\Q}$ is stable under extensions. Thus, the proof of 
\emph{Theorem 
\ref{Del=Nori}} is complete.
\end{cor}

\begin{proof}
The first claim is a consequence of 
Theorem \ref{thm:eff-Del-conj}
and the fact that 
${\sf MHS}_{\leq 1}^{\Q}$ is a thick abelian 
subcategory of ${\sf MHS}^{\Q}_{\eff}$.
The second claim follows from Lemmas
\ref{lem:iff-cond-equiv} and \ref{red:to-Del-ext}.
\end{proof}

We close the paper with the following result.

\begin{thm}\label{left-adj-EHM}
The inclusion 
$\M_1^{\Q}\hookrightarrow \M^{\Q}$ has a left adjoint, denoted 
$(-)_{\leq 1}$. Moreover, the following square is commutative
$$\xymatrix{\M^{\Q} \ar[r] \ar[d]^-{(-)_{\leq 1}} & {\sf MHS}^{\Q}_{\eff} \ar[d]^-{(-)_{\leq 1}}\\
\M^{\Q}_1 \ar[r] & {\sf MHS}_{\leq 1}^{\Q}.\!}$$
\end{thm}

\begin{proof}
We will show that for every effective Nori motive $M$, there exists a 
map $M\to (M)_{\leq 1}$ to an object in 
$\M_1^{\Q}$ which realizes to the analogous map 
for 
effective mixed Hodge structures. This will implies that 
$M\to (M)_{\leq 1}$ is also universal and the theorem 
will follows.
There will be a considerable overlap with 
the proof of Theorem
\ref{thm:eff-Del-conj}.

First, we consider the case of $\widetilde{H}_i(X,Y;\Q)$ for 
$(X,Y,i)\in {}^D(Sch_k)$. 
We use the map \eqref{HiXYQN} constructed in the proof of Theorem 
\ref{thm:eff-Del-conj}. Its Hodge realization
is given by 
\eqref{Hodge-R-Hi-to-N}: so we are done.

Now, let $M$ be a general effective Nori motive
and write $M=M'/M''$ with $M'\subset M''\subset 
\widetilde{H}_i(X,Y;\Q)$. 
As in the proof of Theorem
\ref{thm:eff-Del-conj},
we consider 
$(M)_{\leq 1}=(M')_{\leq 1}/(M'')_{\leq 1}$ where 
$(M')_{\leq 1}$ and $(M'')_{\leq 1}$ are the images 
of $M'$ and $M''$ in 
$(\widetilde{H}_i(X,Y;\Q))_{\leq 1}$. By construction, 
this map realizes to the projection
$R^{\rm Hdg}(M)\to R^{\rm Hgd}(M)_{\leq 1}$. This finishes the 
proof of the theorem.
\end{proof}

\subsection*{\it Acknowledgments} We would like to thank M. Nori for providing some unpublished material regarding his work. We also thank D. Arapura for some helpful conversations on Nori's work.

\end{document}